\documentclass[11pt]{article}

\usepackage{graphicx,,psfrag}

\usepackage{amssymb,amsmath,amsthm,enumerate}
\usepackage{fullpage}
\usepackage{bbm}
\newcommand{\url}{}

\usepackage{tikz}

\RequirePackage[colorlinks,citecolor=blue,urlcolor=blue]{hyperref}

\newtheorem{lemma}{Lemma}
\newtheorem{remark}{Remark}
\newtheorem{proposition}{Proposition}
\newtheorem{definition}{Definition}
\newtheorem{theorem}{Theorem}
\newtheorem{example}{Example}
\newtheorem{corollary}{Corollary}

\newcommand{\cA}{\mathcal{A}}
\newcommand{\cE}{\mathcal{E}}
\newcommand{\cI}{\mathcal{I}}

\newcommand{\EE}{{\mathbf{E}}}

\newcommand{\dE}{\mathbb {E}}

\newcommand{\dZ}{\mathbb {Z}}
\newcommand{\dN}{\mathbb {N}}

\newcommand{\dR}{\mathbb {R}}
\newcommand{\dC}{\mathbb {C}}

\newcommand{\cB}{\mathcal{B}}
\newcommand{\cG}{\mathcal {G}}

\newcommand{\cP}{\mathcal {P}}

\newcommand{\cM}{\mathcal {M}}
\newcommand{\cC}{\mathcal {C}}
\newcommand{\cS}{\mathcal {S}}

\newcommand{\cH}{\mathcal {H}}

\DeclareMathAlphabet{\mathdutchcal}{U}{dutchcal}{m}{n}
\newcommand{\ekk}{\mathdutchcal{k}}
\newcommand{\cJ}{\mathcal{J}}

\newcommand{\ABS}[1]{{{\left| #1 \right|}}} 

\newcommand{\PAR}[1]{{{\left(#1\right)}}} 

\newcommand{\INT}[1]{{{\left[ \hspace{-1pt} \left[ #1 \right] \hspace{-1pt} \right] }}} 

\newcommand{\AND}{\quad \mathrm{and} \quad}

\newcommand{\1}{1\!\!{\sf I}}
\newcommand{\IND}{\1}
\newcommand{\veps}{\varepsilon}

\newcommand{\SC}{\mathbf{s}}

\newcommand{\Wg}{\mathrm{Wg}}

\newcommand{\gauss}{\mathrm{gauss}}

\newcommand{\NN}{\mathfrak{n}}
\newcommand{\MM}{\mathfrak{m}}
\newcommand{\FF}{\mathfrak{f}}
\newcommand{\BB}{\mathfrak{b}}
\newcommand{\PP}{\mathfrak{p}}
\newcommand{\QQ}{\mathfrak{q}}

\newcommand{\ST}{\mathfrak{s}}
\newcommand{\CL}{\circ\BB}

\newcommand{\HU}{{\mathbf{u}}}

\newcommand{\BU}{{+\SC}}

\definecolor{darkred}{rgb}{0.9,0,0.3}

\begin{document}

\title{Freeness for  tensors}
\author{Rémi Bonnin\thanks{Aix-Marseille Univ, CNRS, I2M, Marseille, France. Email: \href{remi.bonnin@ens.psl.eu}{remi.bonnin@ens.psl.eu}}  \; and \; Charles Bordenave\thanks{Aix-Marseille Univ, CNRS, I2M, Marseille, France. Email:  \href{charles.bordenave@univ-amu.fr}{charles.bordenave@univ-amu.fr}}
}

\maketitle


\begin{abstract}
We pursue the current developments in random tensor theory by laying the foundations of a free probability theory for tensors and establish its relevance in the study of random tensors of high dimension. We give a definition of freeness associated to a collection of tensors of possibly different orders. Our definition reduces to the usual freeness when only tensors of order $2$ are concerned.
We define the free cumulants which are associated to this notion of tensor freeness.
We prove that the basic models of random tensors are asymptotically free as the dimension goes to infinity.
On the way, we establish Schwinger-Dyson loop equations associated to random tensors. 
\end{abstract}

\section{Introduction}

Since the fundamental works of Voiculescu \cite{MR1094052,MR1217253}, free probability theory has led to numerous successes notably in operator algebra and random matrix theory, we refer to \cite{MR1775641,MR2760897,zbMATH06684673}. This theory has found many applications in science and engineering, to cite a few, see \cite{MR2884783,https://doi.org/10.1007/s10144-014-0471-0,MR4379548,Couillet_Liao_2022}.   
Tensors  are  versatile algebraic objects which by many aspects are higher order matrices, see the monograph \cite{MR3660696} and references therein. Random tensors play an important and growing role in modern science, notably in quantum field theory \cite{sasakura91,zbMATH06638014,MR4276769,CollinsGurauLionni_1,CollinsGurauLionni_2,gurau2024quantumgravityrandomtensors} and in data analysis and machine learning, see e.g. \cite{NIPS2014_b5488aef,MR4132641,zbMATH07829141,NEURIPS2023_b14d76c7,pandey2024introduction,kunisky2024tensor}. The goal of this work is to pursue the current developments in random tensor theory by laying the foundations of a free probability theory for tensors and establish its relevance in the study of random tensors of high dimension.
The main contributions of this work are the following: 

\begin{enumerate}[-]
\item We give a definition of freeness associated to a collection of tensors of possibly different orders. Our definition reduces to the usual freeness when only tensors of order $2$ are concerned.

\item We define the free cumulants which are associated to this notion of tensor freeness.

\item We prove that the basic models of random tensors are asymptotically free as the dimension goes to infinity.

\item We establish Schwinger-Dyson loop equations associated to random tensors. 

 \end{enumerate}

In the remainder of this introduction, we convey the main notions on tensors underlying our work by using the parallel with matrices (formal definitions are postponed to  Section \ref{sec:1}). We also state our main results. 

\subsection{Distribution of tensors} 

\paragraph{Tensors. } A tensor of order $p \geq 1$ and dimension $N\geq 1$ is commonly defined as an element of the vector space $\cE^N_p = \dC^N \otimes \cdots \otimes \dC^N$, that is the tensor product of $p$ copies of $\dC^N$. The different copies of $\dC^N$ are called the legs of the tensor. By choosing an orthonormal basis of $\dC^N$, a tensor $T$ can be represented by a multi-dimensional array $T = (T_{i_1,\ldots,i_p})$ with $i_t \in \INT{N} = \{1,\ldots,N\}$ for each leg $t \in \INT{p}$.  If $(I,J)$ is a partition of $\INT{p}$, since $(\dC^N)^* = \dC^N$, a tensor defines a linear map from $\dC^J$ to $\dC^I$ whose associated matrix is $T_{I,J} = (T_{i,j})_{i \in \INT{N}^I,j \in \INT{N}^J}$. From this perspective, a central feature of tensors is that they can be contracted along legs: two tensors $T_k \in \cE^N_{p_k}$,  $k=1,2$,  and two subsets of legs $J_k \subset \INT{p_k}$ in bijection of size $|J_1| = |J_2|  = q$ define a new tensor  
$ S  \in \cE^N_{p_1 + p_2 - 2q} $ by considering the matrix product $S = (T_1)_{J^c_1,J_1} (T_2)_{J_2,J_2^c}$ with $J^c_k = \INT{p_k} \backslash J_k$. For example, if $M  = (M_{ij}) \in \cE^N_2$ and $f = (f_i) \in \cE_1^N$, the contraction with respect to the last leg is simply $Mf$, the image of $f$ by the matrix $M$. Similarly, if $M_1,M_2 \in \cE^N_2$, their contraction on the last leg of $M_1$ and first leg of $M_2$ is the usual matrix product $M_1 M_2$ and so on.

\paragraph{Maps of tensors. } These contractions are conveniently represented by symbolic operations on maps where a map is a finite graph  where each vertex has an order among its neighboring edges and with boundary edges or half-edges (we postpone the formal definitions to Section \ref{sec:1}). A tensor $T$ of order $p$ is represented by a map with a single vertex with $p$ boundary edges. 
A contraction of two tensors $T_1$ and $T_2$ is a map with two vertices obtained by forming edges between the contracted legs of the tensors. More generally, let $\MM$ be a map with vertex set $V$, edge set $E$ and with $q\geq 1$ boundary edges say $\partial = (e_1,\ldots,e_q)$. Then if $(T_v)_{v \in V}$ is a collection of tensors where the order of $T_v$ is the degree of $v$ in $\MM$, we can define the tensor in $\cE^N_q$, for $i \in \INT{N}^{\partial}$, 
\begin{equation}\label{eq:trace1}
\MM ( (T_v)_{ v\in V} )_{i_{\partial}} =  \sum_{i \in \INT{N}^{E}} \prod_{v \in V} (T_v)_{i_{\partial v}},
\end{equation}
where $\partial v$ is the sequence of neighboring  edges and boundary edges of $v$. For example, if $\MM$ is a connected map which is a line segment of length $k$ with vertex set $\INT{k}$ and boundary edges attached to vertex $1$ and vertex $k$, then 
$
\MM ( M_1,\ldots,M_k) =  M^{\veps_1}_1 \cdots M_k^{\veps_k},
$
where $M_v^{\veps_v}$ is either $M_v$ or its transpose  $M_v^\intercal$ depending on the order whether of the neighboring edge of $v$. The linear combination of these maps of tensors of possibly different order  encode all possible ways to contract the tensors and can be thought as the extension of the matrix polynomials  in the matrix case.

\paragraph{Trace invariants. } The combinatorial complexity of these maps of tensors is daunting. However, as for matrices, there is a notion of trace. It is defined by using maps $\MM$ without boundary edge which we will call trace maps. Then  if $(T_v)_{v \in V}$ is a collection of tensors where the order of $T_v$ is the degree of $v$ in $\MM$, we can define the scalar by the same formula as in \eqref{eq:trace1}:
\begin{equation}\label{eq:trace2}
\MM( (T_v)_{ v\in V} )  =  \frac{1}{N^\gamma} \sum_{i \in \INT{N}^{E}} \prod_{v \in V} (T_v)_{i_{\partial v}},
\end{equation}
where $\gamma$ is the number of connected components of $\MM$. The normalization is chosen so that if $\MM$ is the trace map with a single vertex and a loop edge, then for $M \in \cE^N_2$, 
$$
\MM( M  ) = \frac 1 N  \mathrm{Tr} (M).
$$
The application $\MM : (T_v)_{v \in V} \to \dC$ is multi-linear and importantly, it is orthogonal  invariant in the following sense. If $T \in \cE^N_p$ and $U \in \cE_2^N$ is an orthogonal matrix, define $T\cdot U^p \in \cE^N_p$ as the contraction of each leg of $T$ by $U$, that is for $j \in \INT{N}^p$
\begin{equation}\label{eq:defTU}
    (T\cdot U^p)_j = \sum_{i \in \INT{N}^p} T_i \prod_{k=1}^p U_{j_k i_k}.
\end{equation}
In other words, we have $T\cdot U^p = \MM ( (T,U,\ldots ,U) ) $ where $\MM$ is an elongated star map, with $T$ in the middle and $U$ on each branch, the second neighboring edge of $U$ being attached to $T$. If $M \in \cE^N_2$, then $M \cdot U^2 = U M U^{\intercal}$. Then, it is straightforward to check that for any orthogonal matrix $U$ and any trace maps,
$$
\MM ( (T_v \cdot U^{p_v} )_{ v\in V} ) = \MM ( (T_v)_{ v\in V} ),
$$
where $p_v$ is the degree of $v$ (and the order of $T_v$). The trace maps $\MM$ form a basis of orthogonal invariant multi-linear application.  We refer to \cite{zbMATH06638014,kunisky2024tensor} for an introduction on these trace invariants (for tensors of even order, it is also possible to define maps which are unitary invariant).

\paragraph{Distribution of tensors. } Exactly as in matrix algebras, we may use the trace invariants $\MM$  to define the distribution of a collection of tensors $\cA  = \{ A_1,\ldots, A_n \}$ of possibly different order, $A_k \in \cE^N_{p_k}$. The distribution of $\cA$ is the collection of all trace maps $\MM ( (T_v)_{v \in V})$ with $T_v \in \cA$ and $\MM$   with compatible degrees.

In this work, we consider a sequence $\cA^N$ of such collection of tensors and we study the limit distribution as the dimension $N $ grows. Our main goal is to characterize such limits when some tensors are random.

\paragraph{Freeness. } It was discovered by Voiculescu  \cite{MR1094052,MR1217253} that the asymptotic distributions of random matrix algebras are captured by the notion of freeness. In Section \ref{sec:1}, we will give a definition of freeness for  the distribution of a finite collection  $(\cA_c)_{c \in \mathcal \cC}$ of tensors where $\cA_{c} = \{A_{c,1}, \ldots, A_{c,n_c}\}$ is a set of tensors of possibly different order. Importantly if the collection  $(\cA_c)_{c \in \mathcal \cC}$ is free then their joint distribution is characterized by the individual distribution of  $(\cA_c)_{c \in \mathcal \cC}$.

The definition of freeness will be given in Section \ref{sec:1}. It relies on the operation on maps which switch two edges while increasing the number of connected components. This defines a poset (partially ordered set) which will play the role of the poset of non-crossing partitions in free probability. We will define the associated free cumulants  in Section \ref{sec:cumul} and notably give an alternative characterization of freeness thanks to the free cumulants.

\subsection{Asymptotic freeness}

We now state our main results pertaining to the asymptotic distribution of random tensors.

\paragraph{Symmetric random tensors. } For a given $p \geq 1$, the symmetric group $\mathrm{S}_p$ acts on $\INT{N}^p$ by permutation of indices: for $i \in \INT{N}^p$ and $\sigma \in \mathrm{S}_p$, $i_\sigma =  (i_{\sigma(1)},\ldots,i_{\sigma(p)})$. For $i,j \in \INT{N}^p$, we say $i \stackrel{p}{\sim} j$ if $i = j_{\sigma} $ for some $\sigma \in \mathrm{S}_p$.

We consider $X= (X_i)_{i \in \INT{N}^p} \in \cE^N_p= (\dC^N)^{\otimes p}$ such that $X_i = X_j$ if $i \stackrel{p}{\sim} j$ and the random variables $(X_i)_{i \in \INT{N}^p / \stackrel{p}{\sim}}$ are  independent, real,
\begin{equation}\label{eq:varX}
 \dE X_i = 0 \quad \AND \quad \dE X_i^2 = \frac{p}{\cP_i},
\end{equation}
where $\cP_i$ is the number of elements in the equivalence class of $i$. Remark that we can equivalently write $\dE X_i^2 = \frac{1}{(p-1)!}\prod_{j=1}^N c_j(i)!,$
where $c_j(i)$ is the number of occurrences of $j$ in $i$. 
We also assume that for each integer $N$, the law of $X_i$ depends only on the equivalence class of $i \in \INT{N}^p$ with respect to the action of $\mathrm{S}_N$. That is $i \underset{N}{\sim} j$ if there exists $\sigma \in \mathrm{S}_N$ such that $(\sigma(i_1),\ldots,\sigma(i_p)) = j$. The law of $X_i$ may depend on $N$.

The main example is the Gaussian Orthogonal Tensor Ensemble (GOTE) where $X$ is Gaussian. In general, we will assume that the moments of $X_i$ are bounded. More precisely,
\begin{enumerate}
    \item [(X1)] For all integers $k \geq 2$, there exists a constant $c(k) > 0$ such that for all integers $N\geq 1$ and $i \in \INT{N}^p$:\begin{equation*}\label{eq:Xkbd} \dE |X_i|^k  \leq c(k). \end{equation*}
\end{enumerate}

We define the normalized symmetric random tensor as 
$$
W^N  = \frac{X}{N^{\frac{p-1}{2}}}.
$$
(Note that for vectors, $p=1$, there is no scaling). The random tensor $W^N$ is the tensor analog of the real Wigner matrices.

Remark that when $p$ is even, we can also define Hermitian complex-valued tensors. Indeed, for $p = 2l$ even, we can identify an element of $\cE^N_p = (\dC^N)^{\otimes p}$ as linear map from $\dC^l$ to $\dC^l$. We may thus distinguish inputs and outputs between the legs of an element in $\cE^N_p$. For simplicity, we will however restrict ourselves to the real case in these notes.

For $p = 1$, the law of large numbers asserts the distribution of $W^N$ converges a.s. (the only   connected map without boundary edges to be considered is the map with $2$ vertices of degree $1$). For $p=2$, the convergence in distribution is the content of Wigner semi-circular Theorem, see the monographs \cite{MR2567175,MR2760897,zbMATH06684673}. For $p \geq  3$, the convergence is due to Gurau \cite{zbMATH06638014,gurau2020generalization} in the real Gaussian case and Bonnin \cite{bonnin2024universality} in the general real case. The limit distribution will be given in Subsection \ref{subsec:asymptfree}.

 We will consider two basic types of random elements and inquire about their asymptotic freeness with respect to $\cA^N_0$.

\paragraph{Assumptions on $(A_i^N)_{i \in I}$. } We consider a finite and deterministic collection $\cA^N_0 = (A^N_i)_{i \in I}$ of elements in $\cE^N$ such that for all $N$, $i$, we have $A^N_i \in \cE^N_{\ell(i)}$. We also assume that $\cA^N_0$ is stable by taking entry-wise complex conjugation: for all $i \in I$, $\bar A_i^N \in \cA^N_0$. We will consider two types of assumptions for the collection  $\cA^N_0 = (A^N_i)_{i \in I}$.  We state both assumptions  here and we will rephrase them in more formal terms in Section \ref{sec:SD} when the proper formalism will be introduced.

\begin{enumerate}
    \item[(A1)] For all trace maps $\MM$, all $(T^N_v)_{v \in V}$ where $T^N_v \in \cA_0^N$ is a tensor of order the degree of $v$, there exists a constant $C(\MM)$ such that for all $N \geq 1$
\begin{equation*}
\ABS{\MM ((T^N_v)_{v\in V})} \leq C(\MM),
\end{equation*}
\end{enumerate}

The second assumption is the same as assumption (A1) except that we consider hyper-maps instead of maps. Loosely speaking, an hyper-map is a map where edges can connect more than two vertices. Definition of trace maps in \eqref{eq:trace2} extends verbatim to trace hyper-maps.

\begin{enumerate}
    \item[(A2)] For all trace hyper-maps $\MM$, all $(T^N_v)_{v \in V}$ where $T^N_v \in \cA_0^N$ is a tensor of order the degree of $v$, there exists a constant $C(\MM)$ such that for all $N \geq 1$
\begin{equation*}
\ABS{\MM ((T^N_v)_{v\in V})} \leq C(\MM),
\end{equation*}
\end{enumerate}

Alternatively, as pointed out by a referee, assumption (A2) can be formulated in terms of trace maps (no hypermaps) as in (A1) but where $T^N_v \in \cA_0^N \sqcup \mathcal \{ \sum_i e_i \otimes \cdots \otimes e_i   \in \cE^N_q: q \geq 3 \}$.

\paragraph{Asymptotic freeness for Wigner tensors. } 
We have the following asymptotic freeness result which extends a classical result for random matrices unveiled by Voiculescu, see \cite{MR2760897,zbMATH06684673}. On our way, we will give another proof of the convergence of the distribution of $W^N$ for $p \geq 3$ in the Gaussian case.

\begin{theorem}\label{th:free1}
If (A1) holds, in the real Gaussian case, the families $\cA^N_0$ and $\{ W^N\}$ are asymptotically free in probability.
\end{theorem}

Informally, asymptotic freeness means that the trace maps in the collection $\cA^N_0$ and $\{ W^N\}$ are asymptotically equivalent to the trace maps computed by assuming that $\cA^N_0$ and $\{ W^N\}$ are free. The proof will rely on the Schwinger-Dyson loop equation that is satisfied by random tensors which have been often used in the matrix case, we refer notably to \cite{VOICULESCU1999101,MR2760897}.  The Schwinger-Dyson loop equations have emerged as a central piece of modern free probability theory.  In our setting, these equations will characterize the limit distributions of $\{\cA^N_0,W^N\}$. 

In general, we believe that Theorem \ref{th:free1} is true also in the non-Gaussian case under an appropriate moment assumption on $X_i$. Our best claim is this direction is the following.

\begin{theorem}\label{th:free1b}
If (X1) and (A2) hold, the families $\cA^N_0$ and $\{ W^N\}$ are asymptotically free in probability.
\end{theorem}

The proof of Theorem \ref{th:free1b} is by comparison with the Gaussian case. We will show that the distribution of a trace map depends at first order only on the first two moments of the Wigner random tensor.

As a consequence of Theorem \ref{th:free1} and Theorem \ref{th:free1b} and independence, we obtain the asymptotic freeness in probability  of independent Wigner tensors of possibly different order. Indeed, Assumptions  (A1) and (A2) hold in probability for Wigner tensors with bounded moments (see proofs in Section \ref{sec:1} or \cite{bonnin2024universality}).

\begin{corollary}
Let $n \geq 1$ be an integer and  $(W^N_1,\ldots,W^N_n)$ be independent Wigner tensors of possibly different orders such that (X1) holds for each $n$. The tensors $(W^N_1,\ldots,W^N_n)$ are asymptotically free in probability.
\end{corollary}

Theorem \ref{th:free1} and Theorem \ref{th:free1b} apply notably to the case where $\cA^N_0$ is a finite collection of vectors. This allows to retrieve in principle \cite[Theorem 2]{JMLR:v23:21-1038} and the main result of \cite{MR4650897}.

\paragraph{Haar unitary and Haar orthogonal matrices. } We take $U^N$ to be Haar distributed on the unitary group $\mathrm{U}(N)$ or the orthogonal group $\mathrm{O}(N)$. The convergence in distribution of $(U_N,U_N^*)$ is due to Voiculescu \cite{MR1601878}.   We have the following asymptotic freeness result which extends a classical result for matrices, see \cite{MR1601878,MR2760897,zbMATH06684673}. 

\begin{theorem}\label{th:free2}
If (A1) holds, the families $\cA^N_0$ and $\{ U_N,U_N^*\}$ are asymptotically free in probability.
\end{theorem}

Again, the proof will rely on the Schwinger-Dyson equation that is satisfied by a random unitary matrix.  Theorem \ref{th:free2} has an interesting corollary for unitary invariant random tensors. More precisely, recall the definition $T \cdot U^p$ for $T \in \cE^N_p$ and $U \in \cE^N_2$ in \eqref{eq:defTU}. If $\cA$ is a subset of tensors of possibly different orders, we denote by $\cA \cdot U^{\#}$ the subset of  tensors of the form for some $p$ and some $A \in \cA \cap \cE^N_p$, $A\cdot U^p$.

\begin{theorem}\label{th:free3}
Let $\cA^N_1$ and $\cA^N_2$ be two finite families  of tensors such that $\cA^N_0 = \cA^N_1 \cup \cA^N_2 \cup \bar \cA^N_1 \cup \bar \cA^N_2$ satisfies (A1). The families $\cA^N_1$ and $\cA^N_2 \cdot U_N^{\#}$  are asymptotically free in probability.
\end{theorem}

\subsection{Organisation of the paper} 
In Section \ref{sec:1}, we will introduce the main definitions, notably maps and freeness. We will also prove the first basic results related to our formalism. In Section \ref{sec:cumul}, we introduce the free cumulants and prove that they characterize freeness. We will also establish a central limit theorem for sums of free tensors. In Section \ref{sec:SD}, we prove the results of asymptotic freeness stated in this introduction. The proof goes by establishing Schwinger-Dyson loop equations adapted to our formalism.

\subsection{Acknowledgements}
The authors thank the two anonymous referees for their detailed and constructive reports. Their comments led us to clarify and improve the presentation of the paper.

\section{Freeness for graphical actions}
\label{sec:1}

\subsection{Combinatorial maps}

For integer $n \geq 1$, we set $\INT{n} = \{1,\ldots,n\}$. 

A {\em combinatorial map} is a finite graph $\MM$ equipped with an order of edges attached to each vertex. More precisely,  for $m$ even integer, a combinatorial map with $m/2$ edges  and $n$ vertices has vertex set $V(\MM) = \INT{n}$ which is encoded by a pair $\MM = (\pi,\alpha)$ of permutations in $S_{m}$. The set $\vec E (\MM) = \INT{m}$ are the directed edges (or half-edges), $\pi$ has $n$ cycles ordered by least elements which are the directed edges attached to each vertex and $\alpha$ is an involution without fixed point whose $m/2$ cycles of length $2$ are  identified as $E(\MM)$ the edges of $\MM$. 
For $v \in V(\MM)$, we denote by $\partial v = (e_1,\ldots,e_p) \in \vec E(\MM)^p$ the cycle of $\pi$ associated to $v$. We always choose $e_1$ such that $e_1 = \min \partial v$. The {\em degree} of $v \in V(\MM)$, $\deg(v)$ is the length of the cycle, that is $p$. We denote by $\cM_0$ the set of combinatorial maps.  

\begin{figure}[h!]
    \centering
    \begin{tikzpicture}[scale=0.8]
    \filldraw[gray] (1.5,0) circle (3pt);
    \filldraw[gray] (0,-2) circle (3pt);
    \filldraw[gray] (3,-2) circle (3pt);
    \draw (1.5,0) .. controls (0.5,1.5) and (2.5,1.5) .. (1.5,0);
    \draw (1.5,0) .. controls (1.2,-0.2) and (0.4,-0.2) .. (0,-2);
    \draw (1.5,0) .. controls (1.5,-1) and (1,-1.5) .. (0,-2);
    \draw (0,-2) -- (3,-2) ;
    \draw (1.5,0) -- (3,-2) ;
    \draw (1.7,0.6) node {\footnotesize $1$};
    \draw (1.3,0.6) node {\footnotesize $2$};
    \draw (0.8,-0.3) node {\footnotesize $3$};
    \draw (1.4,-0.5) node {\footnotesize $4$};
    \draw (2,-0.5) node {\footnotesize $5$};
    \draw (0.6,-1.5) node {\footnotesize $6$};
    \draw (0.1,-1.4) node {\footnotesize $7$};
    \draw (1,-2) node {\footnotesize $8$};
    \draw (2.6,-1.4) node {\footnotesize $9$};
    \draw (2.4,-2) node {\footnotesize $10$};
\end{tikzpicture}
\caption{$\MM=(\pi=(1,2,3,4,5)(6,7,8)(9,10),\alpha=(1,2)(3,7)(4,6)(5,9)(8,10)) \in \cM_0$.} \label{fig:1_map}
\end{figure}

We will also need to introduce combinatorial maps with boundaries. They are conveniently described by a pair $\MM = (\pi,\alpha)$ of permutations in $S_{m}$ where $\alpha$ is an involution and $m$ is not necessarily even. The elements $e \in \vec E( \MM) = \INT{m}$ such that $\alpha(e) = e$ are the boundaries of the maps. The boundaries are naturally ordered by the lexicographic order. For integer $q$, we denote by $\cM_q$ the set of combinatorial maps with $q$ boundary edges and by $\cM = \sqcup_q \cM_q$ the set of all combinatorial maps.

Three very simple maps will appear in the sequel, see Figure \ref{fig:3_maps}. For $q \geq 1$, the {\em star} map $\ST_q \in \cM_q$ is the map with a single vertex and $q$ boundary edges. For $p = 2t$ even, the {\em bouquet} map with $t$ loops $\BB_p = (\pi,\alpha) \in \cM_0$ is the map with a single vertex and $p$ directed edges : $\pi  =(1,\ldots,p)$ is a cycle and $\alpha = (1 , 2) \cdots (p-1 ,p)$. More generally, we call bouquet map a map of the form $\BB^\tau_p = (\pi,\tau)$ for $\pi  =(1,\ldots,p)$ and $\tau \in S_p$ a pairing.
The {\em melon} map of degree $p$ (or {\em Frobenius pair} in the terminology of \cite{kunisky2024tensor}), $\mathfrak{f}_p = (\pi,\alpha) \in \cM_0$ is the map with two vertices and $p$ edges between them : $\pi  =(1,\ldots,p)(p+1,\ldots,2p)$ and $\alpha=(1,p+1)\cdots(p,2p)$. Similarly, for $\sigma \in \mathrm{S}_p$, $\FF_p ^\sigma =(\pi,\alpha^\sigma)$ with $\alpha^\sigma = (1,p+\sigma(1))\cdots(p,p+\sigma(p))$.

\begin{figure}[h!]
    \centering
    \begin{tikzpicture}[scale=0.5]
    \filldraw[gray] (0,0) circle (3pt);
    \draw (0,0) -- (2,0) ;
    \draw (0,0) -- (0.618,1.902) ;
    \draw (0,0) -- (0.618,-1.902) ;
    \draw (0,0) -- (-1.618,1.176) ;
    \draw (0,0) -- (-1.618,-1.176) ;

    \filldraw[gray] (6,0) circle (3pt);
    \draw (6,0) .. controls (9.236,2.351) and (9.236,-2.351) .. (6,0);
    \draw (6,0) .. controls (9.236,2.351) and (6-1.236,3.804) .. (6,0);
    \draw (6,0) .. controls (9.236,-2.351) and (6-1.236,-3.804) .. (6,0);
    \draw (6,0) .. controls (6-1.236,3.804) and (2,0) .. (6,0);
    \draw (6,0) .. controls (6-1.236,-3.804) and (2,0) .. (6,0);

    \filldraw[gray] (10,0) circle (3pt);
    \filldraw[gray] (14,0) circle (3pt);
    \draw (10,0) .. controls (10+1.33,2) and (10+2.67,2) .. (14,0);
    \draw (10,0) .. controls (10+1.33,1) and (10+2.67,1) .. (14,0);
    \draw (10,0) -- (14,0) ;
    \draw (10,0) .. controls (10+1.33,-2) and (10+2.67,-2) .. (14,0);
    \draw (10,0) .. controls (10+1.33,-1) and (10+2.67,-1) .. (14,0);
\end{tikzpicture}
\caption{Combinatorial maps $\ST_5$, $\BB_{10}$ and $\FF_5$.} \label{fig:3_maps}
\end{figure}

We say that two combinatorial maps $\MM = (\pi,\alpha)$ and $\MM' = (\pi',\alpha')$ are equivalent if they have the same number of directed edges, say $m$, and if there exists a permutation $\sigma \in S_{m}$ such that $\pi  = \sigma \circ \pi' \circ \sigma^{-1}$, $\alpha= \sigma \circ \alpha' \circ \sigma^{-1}$ ($\MM$ and $\MM'$ are equal up to a relabeling) and the order in cycles of $\pi$ and boundaries are preserved. We then write $\MM' = \sigma. \MM $. Note that the cycles of $\pi'$ and $\pi$ are then in bijection by $\sigma$.

There will also be maps with colored vertices. They are formally defined as follows. Let $\cI$ be a set equipped with a map $\ell : \cI \to \{1,2,\cdots\}$. We denote by $\cI^{\MM}$ the set of pairs $(\MM,w)$ such that $\MM \in \cM$ and $w \in \cI^{V(\MM)}$ is such that $w_v \in \{ i : \ell (i) = \deg(v)\}$ for all $v \in V(\MM)$; $w_v$ is then referred as the color of the vertex $v$. Such a pair is called an $\cI$-map. The sets $\cM_q(\cI)$ and $\cM( \cI) = \sqcup \cM_q(\cI)$ are respectively the $\cI$-maps with $q$ boundaries and all $\cI$-maps. Our notion of equivalence extends to $\cI$-maps if we further require that colors of vertices are preserved.





\subsection{Action of combinatorial maps}

\label{susbec:action}
We now consider a complex vector space $\cE_1$ and set $\cE_0 = \dC$, and $\cE_p = \cE_1^{\otimes p}$ for $p \geq 2$. We denote by $\cE = \sqcup \cE_p$ the disjoint union of these vector spaces. We assume that $\cM$ acts on $\cE$ in the following sense. For each $\MM \in \cM_q$ with $n$ vertices, as above we let $\cE_\MM = \{ (x_1,\ldots,x_n) : x_v \in \cE_{\deg(v)} \}$. In the above terminology, if $\MM \in \cM$ and $x=(x_1,\ldots,x_{\vert V(\MM)\vert}) \in \cE_\MM$, then $(\MM,x)$ is an $\cE$-map, where $\cE = \sqcup \cE_p$ is equipped with the map $l : x_v \in \cE \mapsto p$ such that $x_v \in \cE_p$. We assume that there is an application $\cE_\MM \to \cE_q$ which we also denote by $\MM$ with a slight abuse of notation.  This action of $\cM$ has the following properties: 
\begin{enumerate}[]
\item[(CI)] {\em (Class invariance). }  If $\MM$ and $\MM'$ are equivalent, then $\MM((x_v)_{v \in V(\MM)}) = \MM'((x_{\sigma(v)}))_{v \in V(\MM)})$ with $\MM' = \sigma. \MM$.
\item[(M)] {\em (Morphism property). } If the connected components of  $\MM$ are $(\MM_1,\ldots,\MM_\gamma)$, then 
$$
\MM((x_v)_{v \in V(\MM)}) = \bigotimes_{i=1}^\gamma \MM_i( (x_v)_{v \in V(\MM_i)}).
$$
\item[(L)] {\em (Linearity). } For each integer $q \geq 0$ and $\MM \in \cM_q$, the application $\MM : \cE_\MM \to \cE_q$ is multi-linear.
\item[(S)] {\em (Substitution property). } The action is consistent by substitution of sub-maps. More precisely, if $\MM \in \cM_q$ with $k$ vertices and for all $v \in V(\MM)$, $\MM_v \in \cM_{\deg(v)}$, denote $\MM \circ (\MM_1,\ldots,\MM_k) \in \cM_q$ the map where each vertex $v$ of $\MM$ has been replaced by $\MM_v$ with the right matching on the boundary edges, and then
$$
\MM(\MM_1((x_v)_{v \in V(\MM_1)}),\ldots,\MM_k((x_v)_{v \in V(\MM_k)})) = [\MM \circ (\MM_1,\ldots,\MM_k)] ((x_v)_{v \in \cup_{j=1}^k V(\MM_j)}).
$$
Moreover, for all $q \geq 1$, the star map acts as an identity: for all $x \in \cE_q$, $\ST_q(x) = x.$

\item[(Id)] {\em (Existence of an identity in $\cE_p$ for even $p$). } Let $p$ be even and $\tau=(1,2)(3,4)\ldots (p-1,p) \in \mathrm{S}_p$ the canonical pairing. If $\MM \in \cM$ and $u \in V(\MM)$ with degree $p$, if $|V(\MM)| \geq 2$, we denote $\MM^{\backslash u} =(\pi',\alpha')$ the map obtained from $\MM$ as follows. If there is no edge from $u$ to $u$, remove the vertex $u$, denote the boundaries $\alpha( \partial u) = (e_1,\ldots,e_{p})$ and rewire the new  with $\alpha'(e_j) = e_{\tau(j)}$ for any $j \in \INT{p}$. If there is an edge from $(\partial u)_{r}$ to $(\partial u)_{s}$ in $\MM$, then we rather rewire $\alpha'(e_{\tau(r)}) = e_{\tau(s)}$ and do it successively for all edges from $u$ to $u$. 
There exists an element $1_p \in \mathcal E_p$ such that for any $\MM \in \cM_0$, $u \in V(\MM)$ and $x \in \cE_\MM$ such that $x_{u} = 1_p$ we have 
$$\MM( (x_v)_{v \in V(\MM)}) = \MM^{\backslash u} ( (x_v)_{v \in V(\MM) \backslash \{u\}}),$$ 
as soon as $\gamma(\MM^{\backslash u}) =\gamma(\MM) + p/2 -1$ (that is the maximal possible number of connected components). 
We may define also all the permuted identities associated to other pairings $\tau \in S_p$ and denote it $1^\tau_p$.
\end{enumerate}

These axioms are close in spirit to the algebraic traffic spaces developed in  \cite{zbMATH07338275,zbMATH07822923} and can be formalized in terms of operads \cite{zbMATH00970059}. If we restrict ourselves to maps with vertices of degree $2$, a connected map $\MM \in \cM_2$ defines a product of operators and a map $\MM \in \cM_0$ is a linear function.

In this general framework, there is a subtlety associated to the various identities $1^\tau_p$ for even $p \geq 4$ and pairing $\tau$. We will usually require an extra property which can be thought of as an orthonormality condition. 
\begin{definition}[Orthonormal identities]\label{def:normal}
We say that the identities are orthonormal if for all even $p$ and pairing $\tau,\tau' \in \mathrm{S}_p$, we have $\BB_p^{\tau'} (1^\tau_p) = \delta_{\tau',\tau}$ where $\BB_p^{\tau'}$ is the bouquet of $p/2$ loops associated to the pairing $\tau'$. 
\end{definition}

\begin{remark}[Variants] \label{rk:variant} There are natural variants of the above axioms. Let us cite a few. As in $*$-algebras, we could have added antilinear maps $x \mapsto x^*$ on $\cE_p$, $p \geq 1$. Also, the maps $\cM$ could be (edge or vertex) colored maps with colors on some finite set and the action of a map could depend on the colors. We might also assume that there are two classes of directed edges: inputs and outputs. Then the permutation $\alpha$ is constrained to match input and output directed edges.
These variants could be useful for example to describe actions of maps on tensors with legs of different dimensions.
With minor modifications, all these variants can be treated along the lines of the framework we are focusing on, see follow-up Remark \ref{rk:variant2}.
\end{remark}

\subsection{Trace invariants}

\label{subsec:TR}
The central example  is the following. For integer $N \geq 1$, we set $\cE^N_p = (\dC^N)^{\otimes p}$, that is $\cE^N_1$ are vectors, $\cE^N_2$ are matrices, and, for $p \geq 3$, $\cE^N_p$ are tensors with $p$ legs of dimension $N$. For $x \in \cE^N_p$, we write $x = (x_{i})_{i \in \INT{N}^p} \in \cE^N_p$. The action of $\cM$ on $\cE^N = \sqcup \cE^N_p$ is defined for $\MM \in\cM_q$, with $\partial = (e_1,\ldots,e_q) \subset \vec E(\MM)$ being its boundary, as
\begin{equation}\label{eq:tracem}
\MM ( (x_v)_{ v\in V(\MM)} )_{i_{\partial}} = \frac{1}{N^\gamma} \sum_{i \in \INT{N}^{E(\MM)\setminus \partial}} \prod_{v \in V(\MM)} (x_v)_{i_{\partial v}},
\end{equation}
where $\gamma = 0$ if $q \geq 1$ and, for $ q = 0$, $\gamma$ is the number of connected components of $\MM$. For $m \in \cM_0$, they are the natural generalization of trace for matrices.  We refer to \cite{zbMATH06638014,kunisky2024tensor} for an introduction on these trace invariants. For $p=2t$ even, the identity is the tensor $1_p = N^{-t+1} (I_N)^{\otimes t} \in \cE^N_p$ and $\BB^{\tau'}_p ( 1^\tau_p) = \delta_{\tau',\tau} + O(1/N)$. It is thus asymptotically orhonormal in the sense of Definition \ref{def:normal}.

\begin{example}
    For the map $\MM$ given in Figure \ref{fig:1_map} which has no boundary, and $T_1 \in \cE^N_5$, $T_2 \in \cE^N_3$ and $M_1 \in \cE^N_2$,
$$
\MM ( T_1,T_2,M_1 ) = \frac{1}{N} \sum_{a,b,c,d,e} (T_1)_{aabcd} (T_2)_{bce} (M_1)_{de},
$$
Also, we have for instance $\ST_5 (T_1) =T_1$.
\end{example}

\subsection{Action distribution}

We come back to the general setup. Let  $\cA = \{ a_i : i \in \cI\}$ be a subset of elements in $\cE$ and $\ell : \cI \to \{1,2,\cdots\}$ is such that $a_i \in \cE_{\ell(i)}$.  We define the {\em distribution} of $\cA$ as the collection of all $\MM(x)$ for all $\cA$-maps $(\MM,x)$ with $\MM \in \cM_0$ (and $x = ( x_1,\ldots, x_n) \in \cA^n \cap \cE_\MM$ by definition). We say that $(\MM,x)$ is {\em centered} if $\MM(x)  = 0$.

A collection of vectors spaces $\cH = (H_p)_{p  \geq 0}$ with $H_p $ subspace of $\cE_p$ is said to be a {\em $\cM$-bundle} if the following two conditions are met: (a) for all even $p$ and pairing $\tau$, $1^\tau_p \in H_p$  and (b) for all $q \geq 0$, for all $\cH$-maps $(\MM,x)$ with $\MM \in \cM_q$ and $x \in \cH^n \cap \cE_\MM$, we have $\MM(x) \in H_q$.    The {\em distribution} of $\cH$ is the collection of all $\MM(x)$ for all $\cH$-maps $(\MM,x)$ with $\MM \in \cM_0$.

If  $\cA = \{ a_i : i \in \cI\}$ is a subset of elements in $\cE$, we denote by  $\langle \cA \rangle = (H_p)_{p \geq 1}$ the $\cM$-bundle spanned by $\cA$. That is, $H_p$ is spanned by $\MM(x)$ for all $\cA$-maps $(\MM,x) \in \cM_p( \cA)$ and, if $p$ even, $1^\tau_p$ for all pairing $\tau$. Note that by properties (S)-(L)-(Id), the distribution of $\langle \cA \rangle$ is characterized by the distribution of $\cA$.

\subsection{Freeness}

We now introduce a notion of freeness for the distribution of elements in $\cE$. In this section, we assume that we have orthonormal identities as per Definition \ref{def:normal}.

\paragraph{The non-crossing poset on maps. } Let $m$ be an even integer, $\pi \in S_{m}$ and $\cM_{\pi} \subset \cM_0$ be the set of maps $\MM = (\pi,\alpha)$ for some involution $\alpha \in S_{m}$. We consider the graph $\cG_\pi$ on $\cM_{\pi}$ where two maps $\MM = (\pi,\alpha)$  and $\MM' = (\pi,\alpha')$ are connected by an edge if they differ by a switch, that is $\alpha' \alpha $ is the product of two disjoint transpositions. If $\MM$ and $\MM'$ differ by a switch then $\gamma(\MM) - \gamma(\MM') \in \{-1,0,1\}$, where $\gamma(\cdot)$ is the number of connected components of a map. We further say that $\MM < \MM'$ if $\gamma(\MM) = \gamma(\MM') + 1$. We extend this relation $<$ to $\cM_{\pi}$ by transitivity. It gives to $\cM_{\pi}$ the structure of a poset which we will call, by analogy, the non-crossing poset on $\cM_{\pi}$.

Indeed, for maps with vertices of constant degree $2$, we retrieve non-crossing conditions in the following sense. A connected map with vertices of degree $2$ is equivalent to the map $\MM=(\pi,\alpha) \in \cM_0$ with $\pi=(1,2)\ldots(2n-1,2n)$ and $\alpha=(2,3)(4,5)\ldots(2n,1)$. For $\MM' \in \cM_\pi$, denote $\PP(\MM')$ the partition of $\INT{n}$ whose blocks are the vertices of $\MM'$ in the same connected components. Then, $\MM'<\MM$ if and only if $\PP(\MM')$ is a non-crossing partition. 

\begin{figure}[h!]
     \centering
     \begin{tikzpicture}[scale=0.5]
    
    \fill[lightgray,even odd rule] (0,0) circle(1.5) circle(2);
    \fill[lightgray,even odd rule] (5,0) circle(1.5) circle(2);
    
    \draw (0,0) node {$\MM_1$} ; 
    \draw (5,0) node {$\MM_2$} ;
    
    \draw (1,1) .. controls (2,2) and (3,2) .. (4,1);
    \draw (1,-1) .. controls (2,-2) and (3,-2) .. (4,-1);
    
    \draw (9,0) node {$>$} ;
    
    \fill[lightgray,even odd rule] (13,0) circle(1.5) circle(2);
    \fill[lightgray,even odd rule] (18,0) circle(1.5) circle(2);
    
    \draw (13,0) node {$\MM_1$} ; 
    \draw (18,0) node {$\MM_2$} ;
    
    \draw (14.5,1) .. controls (15.5,0.5) and (15.5,-0.5) .. (14.5,-1);
    \draw (16.5,1) .. controls (15.5,0.5) and (15.5,-0.5) .. (16.5,-1);
    
\end{tikzpicture}
\caption{Non-crossing poset. In the map on the left, all edges belong to a single block (connected component). In the map on the right, the edges are split into two blocks, and exchanging two edges from different blocks can no longer increase the number of connected components, yielding the characteristic non-crossing structure.} 
\end{figure}

We say that $\MM$ is {\em minimal} if there is no $\MM'$ such that $\MM' < \MM$. For $\MM \in \cM_0$, we denote $\cP_{\MM}$ the poset of maps $\MM'$ such that $\MM \leq \MM'$. Importantly, the poset is consistent with the action on $\cE$. More precisely, if $\MM \leq \MM'$ then $\cE_{\MM'} = \cE_{\MM}$. In other words, the set $\cE_{\MM}$ is common on the poset  $\cP_{\MM}$.

\paragraph{Definition of freeness. } Assume orthonormal identities as per Definition \ref{def:normal}. Let $(\cA_c)_{c \in \cC}$ be a finite collection of  subsets in $\cE$ and set $\cA = \sqcup_{c \in \cC} \cA_c = \{ a_i : i \in \cI\}$ be the disjoint union of these subsets. Let $\langle \cA \rangle$ and  $\langle \cA_c \rangle$ be the $\cM$-bundles spanned by $\cA$ and $\cA_c$, $c \in \cC$. 

Let $(\MM, x)$ be an $\langle \cA \rangle$-map, we say that it is {\em monochromatic} if there is a $c \in \cC$ such that $x_i \in \langle \cA_c \rangle $ for all $i \in V(\MM)$. 
Also we say that two $\cA$-maps $(\MM',x)$ and $(\MM,x)$ such that $\MM'<\MM$ differ by a {\em chromatic switch} if the switch changes two edges in $\MM$ between vertices, say $(u_1,u_2)$ and $(u_3,u_4)$, with $x_{u_i} \in \langle \cA_{c_i} \rangle$ such that $c_1 \ne c_2$, $c_3 \ne c_4$ but $\{c_1,c_2\} \cap \{c_3, c_4\} \ne \emptyset$. 

For $(\MM,x) \in \cM_0 (\langle \cA \rangle)$, we define the maps $\widehat \MM \leq \MM$ as the ones satisfying:
    \begin{itemize}
    \item[(P1)] Each connected component of $\widehat \MM$ is monochromatic or minimal non-monochromatic,
    \item[(P2)] 
    If $\widehat \MM \ne \MM$, there is a path $\MM = \MM_0 > \MM_1 > \ldots > \MM_t = \widehat \MM$ in $\cM_{\pi}$ with only chromatic switches (that is, for all $1\leq i\leq t$, $\MM_{i-1}$ and $\MM_{i}$ differ by a chromatic switch) and for all $1 \leq i \leq t$, $\MM_{i-1}$ does not satisfy (P1).
\end{itemize}
When all the elements are matrices, a connected map $\MM$ is a cycle and the maps $\widehat \MM$ are the non-crossing partitions where consecutive elements of the same family must belong to the same block.

We are ready for the main definition.

\begin{definition}[Tensor freeness]
\label{def:freeness}
    The sets $(\cA_c)_{c \in \cC}$ are {\em free} if for all $\langle \cA \rangle $-maps $(\MM,x) \in \cM_0(\langle \cA \rangle)$ we have $\MM(x) = 0$ as soon as any map $\widehat \MM$ satisfying (P1)-(P2) has a minimal non-monochromatic or a centered monochromatic connected component.
\end{definition}

\begin{remark}[Variants]
    \label{rk:variant2} For the variants described briefly in Remark \ref{rk:variant} the definitions of distributions and freeness carry over. For actions with inputs and outputs, we should simply require that the switches are only between allowed edges (two inputs and two outputs). When there are antilinear maps $x \mapsto x^*$ on $\cE_p$ for all $p\geq 1$, when we define the distribution of a set $\cA$ or a $\cM$-bundle $\cH$, we should require that there are $*$-invariants.
\end{remark}

\subsection{Freeness and individual distributions}
\label{subsec:indiv}

The goal of this section is to show that freeness has the expected determinacy property: the distribution of a union of free families is determined by the individual distributions of the families.

\begin{theorem}\label{th:individualdistrib}
If the sets $(\cA_c)_{c \in \cC}$ are free, then the distribution of $\cA =  \sqcup_{c \in \cC} \cA_c $ is characterized by the individual distributions of $\cA_c$, $c\in \cC$. 
\end{theorem}

To prove Theorem \ref{th:individualdistrib}, we first establish several preliminary results.

\paragraph{Maximal colored sub-maps.} 
Consider families $(\cA_c)_{c \in \cC}$ with $|\cC| \geq 2$ and set $\cA = \sqcup \cA_c$. In an $\cA$-map, $(\MM,x)$ we call an edge chromatic if its endpoints are decorated by elements belonging to different families. The maximal colored sub-maps of $(\MM,x)$ are the connected components obtained from $\MM$ by deleting all chromatic edges. Equivalently, two vertices are in the same maximal colored sub-map if they can be joined by a path all of whose vertices are decorated by elements of the same family. We denote $(\MM_1,x),\ldots,(\MM_k,x)$ the corresponding maximal colored sub-maps, where $\MM_j \in \mathcal{M}_{q_j}$ and $(\MM_j,x)$ is monochromatic (in one of the $\cA_c$'s). Then we may write
\begin{equation}\label{eq:unionmap}
     \MM(x) = [\MM_{\BB} \circ (\MM_1,\ldots,\MM_k)](x),
\end{equation}
for some $\MM_{\BB} \in \mathcal{M}_q(\MM)$. By maximality of the colored sub-maps, all edges of $\MM_{\BB}$ are chromatic.

If for some $j$, $\MM_j \in \cM_{q_j}$ with $q_j$ even, we set $\MM_j^{\CL}= \BB_{q_j} \circ \MM_j  \in \cM_0$, or in other words, if $\MM_j = ( \pi,\alpha)$, then $\MM_j^{\CL} = ( \pi,\alpha^{\CL})$ where we rewire the boundary edges together, that is $\alpha^{\CL}(e_i)=e_{i+1}$ for odd $i$. 
For $\tau \in \mathrm{S}_{q_j}$ a pairing, we similarly define $\MM_j^{\CL^\tau}= \BB_{q_j}^\tau \circ \MM_j  \in \cM_0$, where we rewire the boundary edges with $\alpha^{\CL}(e_i)=e_{\tau(i)}$.

\begin{lemma}\label{le:compcolor}
    Let 
    $\MM = \MM^1 > \MM^2 > \ldots > \MM^T $
    be a chain of chromatic switches.
    For $1\leq t < T$, denote
\begin{align*}
    k_t &:= \text{the number of initial colored blocks which are still maximal colored sub-maps of } \MM^t, \\
    a_t &:= \text{the number of chromatic edges of } \MM^t, \\
    c_t &:= \text{the number of non-monochromatic connected components of } \MM^t.
\end{align*}
    Then for every $1 \leq t < T$,
    $$ \frac{k_{t+1}}{2} - a_{t+1} + c_{t+1} \geq \frac{k_{t}}{2} - a_{t} + c_{t} +1 .$$
\end{lemma}

\begin{proof}
Consider the chromatic switch from $\MM^t$ to $\MM^{t+1}$. Let the two switched chromatic edges be $(u_1,u_2)$ and $(u_3,u_4)$, we first consider the case where $x_{u_1},x_{u_3}\in\langle\cA_1\rangle$, and $x_{u_2},x_{u_4}\in\langle\mathcal A_2\rangle$.

First suppose that the switch is of type ``$-$'', namely $(u_1,u_2),(u_3,u_4) \longmapsto (u_1,u_4),(u_2,u_3)$. The two new edges are still chromatic so $a_{t+1}=a_t$. Moreover, no two same-colored components are joined, so $k_{t+1}=k_t$. Since the switch is strict, it disconnects one connected component and each of the two new components contains one of the new chromatic edges, and is
therefore non-monochromatic, thus $c_{t+1}=c_t+1$.

Now suppose that the switch is of type ``$+$'', namely $(u_1,u_2),(u_3,u_4) \longmapsto (u_1,u_3),(u_2,u_4)$. The two new edges are non-chromatic, so $a_{t+1}=a_t-2$.
\newline First, if both pairs $(u_1,u_3)$ and $(u_2,u_4)$ already belonged to the same
colored components, then no initial colored block is lost, $k_{t+1}=k_t$. The switch can decrease the number of non-monochromatic connected components by at most one, so $c_{t+1}\geq c_t-1$.
\newline Second, if exactly one of the pairs $(u_1,u_3)$ and $(u_2,u_4)$ belonged to the same colored component, then at most two initial colored blocks cease to be maximal, therefore $k_{t+1}\geq k_t-2$. Moreover, the number of non-monochromatic connected components does not decrease, hence $c_{t+1}\geq c_t$.
\newline Finally, if neither $u_1$ and $u_3$, nor $u_2$ and $u_4$, belonged to the same colored component, then at most four initial colored blocks cease to be maximal, so $k_{t+1}\geq k_t-4$. Since the switch is strict, the component is disconnected. In this case both new connected components are still non-monochromatic, so $c_{t+1}= c_t+1$.

By definition of a chromatic switch it remains to check the case of a switch involving $(u_1,u_2),(u_3,u_4)$ with $x_{u_1},x_{u_3} \in \langle \cA_1 \rangle$, $x_{u_2} \in \langle \cA_2 \rangle$ and $x_{u_4} \in \langle \cA_3 \rangle$, $\cA_1, \cA_2, \cA_3$ being three distinct families. 
\newline If the switch is of type ``$-$'', we still have $a_{t+1}=a_t $, $k_{t+1}=k_t$ and $c_{t+1} = c_t +1$.
\newline If the switch is of type ``$+$'' (joining $u_1$ and $u_3$), then we now have $a_{t+1}=a_t -1$. There are two possibilities: if $u_1$ and $u_3$ already belonged to the same colored submap before the switch, then $k_{t+1}=k_t$ and $c_{t+1} \geq c_t$. Otherwise, $k_{t+1}=k_t-2$ and $c_{t+1} = c_t +1$.  

The lemma is proved in all cases.
\end{proof}

\begin{proof}[Proof of Theorem \ref{th:individualdistrib}] 
It is enough to consider connected maps. Indeed, from the morphism property (M), the evaluation of a map factorizes over its connected components, and each connected component involves no more maximal colored submaps than the original map. Let $(\MM,x)$ be a connected $\langle \cA\rangle$-map. Let $(\MM_1,x),\ldots,(\MM_k,x)$ be its maximal colored submaps, and write
$$\MM(x)= \bigl[\MM_{\BB}\circ(\MM_1,\ldots,\MM_k)\bigr](x).$$
Assume first that no admissible descendant $\widehat\MM\leq \MM_{\BB}$ satisfying (P1)-(P2) has only monochromatic connected components. Then every admissible descendant has a non-monochromatic connected component, and its contribution vanishes by freeness. Hence $\MM(x)=0$. Equivalently, the only case which can give a non-zero contribution is the case where some admissible descendant has only monochromatic connected components. 

Now assume that $\widehat\MM\leq \MM_{\BB}$ satisfying (P1)-(P2) has only monochromatic connected components.
We claim that at least one connected component of $\widehat \MM$ is a bouquet map. 

We proceed by contradiction. Assume that $\widehat \MM$ has no bouquet connected component. 
Then it has at most $k/2$ connected components. Hence there exists $T \leq k/2$ such that
$$\MM_{\BB} = \MM^1 > \MM^2 > \ldots > \MM^T = \widehat \MM, $$
is a chain of chromatic switches. 
At the end, since all the connected components of $\widehat \MM$ are monochromatic,
$$a_T=0 \quad \text{ and } \quad c_T=0 ,$$
where $a_t,c_t$ are as in Lemma \ref{le:compcolor}.
Moreover, since each chromatic switch removes at most two chromatic edges and all chromatic edges have disappeared at step $T$, we have
$$ a_1 \leq 2(T-1) .$$
Clearly, we also have $k_1=k$, $c_1=1$. Applying Lemma \ref{le:compcolor} along the chain gives
$$ \frac{k_T}{2} = \frac{k_T}{2} - a_T +c_T \geq \frac{k_1}{2} - a_1 +c_1 + T -1 \geq \frac{k}{2} - T + 2 \geq 2 .$$
In particular $k_T >0$. Thus, at least one initial block $\MM_j(x)$ is still a maximal colored submap of $\widehat \MM$. Since $\widehat \MM$ has only monochromatic connected components, then this connected component is precisely a bouquet map. We get our contradiction.

Having established this result, we now proceed by induction on $k$ to prove the theorem.
If $k=1$, the result is clear as $\MM(x)=\widehat \MM(x)$ is monochromatic. Now assume that $k \geq 2$.
We define for $1\leq j \leq k$ such that $\MM_j \in \cM_{q_j}$ with $q_j$ even,
\begin{equation} \label{centering}
     x_j^0 = \MM_j(x) - \sum_{\tau \in \mathcal{P}_2(q_j)} \MM_j^{\CL^\tau}(x) 1_{q_j}^\tau \in \cE_{q_j},
\end{equation}
where $\MM_j^{\CL^\tau}(x)= (\BB_{q_j}^\tau \circ \MM_j)(x) \in \dC$,  $1^\tau_{q_j} \in \cE_{q_j}$ are the identities from property (Id) and $\mathcal{P}_2(q_j)$ is the set of pairings of $\INT{q_j}$. By properties (L), (Id) and orthonormality of the identities, we have for any pairing $\tau_j \in S_{q_j}$,
\begin{align*}
    \BB_{q_j}^{\tau_j} (x_j^0) 
    &= \MM_j^{\CL^{\tau_j}}(x)- \sum_{\tau} \MM_j^{\CL^\tau}(x) \BB_{q_j}^{\tau_j}(1_{q_j}^\tau)  \\
    &= \MM_j^{\CL^{\tau_j}}(x)- \MM_j^{\CL^{\tau_j}}(x) =0 .
\end{align*}
If $q_j$ is odd, we denote $x_j^0 = \MM_j(x) \in \cE_{q_j}$ and with a slight abuse of notation, $\MM_j^{\CL^\tau}(x)=0 \in \dC$ for any $\tau$.
Note that $\MM^{\CL}_j$ is monochromatic, so $x_j^0 \in \langle \cA_c \rangle$, for some $c$. If $q_1$ is even, we can write
\begin{align*}
    \MM(x) &= [\MM_{\BB}\circ (\MM_1,\ldots,\MM_k)](x) \\
    &= \MM_{\BB}(\MM_1(x),\ldots,\MM_k(x)) \tag*{\text{by (S)}} \\
    &= \MM_{\BB}(x_1^0,\MM_2(x),\ldots,\MM_k(x)) + \sum_{\tau} \MM_1^{\CL^\tau}(x)  \MM_{\BB}(1_{q_1}^\tau,\MM_{2}(x),\ldots,\MM_k(x)) \tag*{\text{by (L)}} 
\end{align*}
Note that this final expression is also true in the trivial case where $q_1$ is odd.
In either case, the second term is characterized by the individual distributions of $\cA_c, c \in \cC$ by the induction hypothesis (and it is $0$ if $q_1$ is odd). Then, we continue iteratively this decomposition on all sub-maps $\MM_2 ,\ldots,\MM_k$:
$$
\MM(x) = \MM_{\BB} (x_1^0,\ldots,x_k^0) + \sum_{l= 1}^k   \sum_{\tau} \MM_l^{\CL^\tau}(x) \MM^l_{\BB}\left((x^0_j)_{j < l },1_{q_l}^\tau ,(\MM_j(x))_{j>l}\right).
$$ 
Each term of the right hand side sum is characterized by the individual distribution of  $(\cA_c)_{c \in \mathcal \cC}$ by induction.
It remains to prove that 
$$\mathfrak{o}(x):=\MM_{\BB} (x_1^0,\ldots,x_k^0) = 0.$$ 
Remark that, by (P2), the maximal colored sub-maps of $(\mathfrak{o},x)$ are the singletons $x_1^0,\ldots,x_k^0$. Any $\widehat \MM \leq \MM_{\BB}$ satisfying (P1)-(P2) contains a bouquet map of the form $\BB_{q_j}^{\tau_j}$ which is therefore centered, since the corresponding $q_j$ is necessarily even and $\BB_{q_j}^{\tau_j} (x_j^0)=0$. The result is then a consequence of the definition of freeness. 
\end{proof}



\subsection{Freeness and melonics}

Melonic maps form a particularly important class of maps in the poset
introduced above. We say that a connected map $\MM$ is {\em melonic} if it is greater, in the poset, than a disjoint union of melon maps. Equivalently, if $\MM$ has $2k$ vertices, then there exist degrees $p_1,\ldots,p_k$ and permutations $\sigma_j\in S_{p_j}$ such that
$$\MM \geq \mathfrak f_{p_1}^{\sigma_1} \sqcup \cdots \sqcup \mathfrak f_{p_k}^{\sigma_k}.$$
Thus $\MM$ can be reduced, by a finite sequence of disconnecting switches, to
a disjoint union of $k$ elementary melon maps. This definition is equivalent to the usual recursive construction of
melonic maps. Namely, a connected combinatorial map $\MM$ with $2k$ vertices is melonic if and only if either $k=1$ and $\MM$ is itself a melon, or there exists a switch $\MM' < \MM$ such that $\MM'$ has exactly two connected components, one of which is a melon map $\mathfrak f_p^\sigma$, and the other of which is a melonic map with $2(k-1)$ vertices.

In the planar case, when all elementary melons have the same degree $p$, the number of rooted planar melonic maps with $2k$ vertices is given by the Fuss-Catalan number
$$ F_p(k):=\frac{1}{pk+1}\binom{pk+1}{k}.$$
These maps are in bijection with $p$-uniform plane hypertrees with $k$ hyperedges; see, for instance, \cite{bonnin2024universality}. Equivalently, they are counted by the $(p-1)$-divisible non-crossing partitions of $[k(p-1)]$, that is, non-crossing partitions whose block sizes are multiples of $p-1$.
This bijective picture extends to melonic maps with melons of arbitrary degrees. In that case, one obtains plane hypertrees with $k$ hyperedges, where the $j$-th hyperedge has order equal to the degree of the corresponding melon $\mathfrak f_{p_j}^{\sigma_j}$.

\subsection{Convergence in distribution and asymptotic freeness}
\label{subsec:asymptfree}

We may classically define convergence in distribution and asymptotic freeness from our definitions. We write it formally as follows.

Fix a finite set $\cI$ and a function $\ell : \cI \to \{1,2,\ldots \}$.  For each $N \geq 1$, let $\cE^N = \sqcup_q \cE^N_q$ be a union of vector spaces as above and $\cA^N = \{ a_i^N : i \in \cI \}$ such that $a^N_i \in \cE^N_{\ell(i)}$ for all $N$.

For readability, we will use the notation $\MM = (\pi,\alpha,w)$ to denote an $\cI$-map (in place of $(\MM,w)$ as above).  If $\MM = (\pi,\alpha,w) \in \cM_q(\cI)$ is an $\cI$-map, we 
define $\MM(\cA_N) \in \cE^N_q$ as the image of the action of the associated $\cA_N$-map $\MM$ where for all $v \in V(\MM)$, $w_v$ is replaced by $a^N_{w_v}$. 

We say that the sequence $(\cA^N)_{N \geq 1}$ {\em converges in distribution} if for all $\cI$-maps $\MM \in \cM_0( \cI)$, there exists a number $\mu(\MM)$ such that

\begin{equation}\label{eq:cvdist}
\lim_{N \to \infty} \MM( \cA_N) = \mu (\MM).
\end{equation}

For example, the random real symmetric tensors $W^N$ converges in distribution.  The asymptotic distribution is as follows : $\mu(\MM) = 1$ if the connected components of $\MM$ are melonic and $0$ otherwise, see \cite{zbMATH06638014,gurau2020generalization,bonnin2024universality}. This result can also be easily extracted from the Schwinger-Dyson equation established in Section \ref{sec:SD}.

We now define asymptotic freeness. Let $\cC$ be a finite set and let $(\cA^N_c)_{c \in \cC}$ be a finite collection of disjoint subsets in $\cE^N$. We write $\cA^N_{c} = \{ a^N_{c,j} : j \in \cI_c \}$ and $\cA^N = \sqcup \cA^N_c = \{ a_i ^N : i \in \cI\}$ with $\cI = \sqcup_c \cI_c$. As above, we assume that for each $i \in \cI$, there exists $\ell(i)$ such that $a^N_i \in \cE^N_{\ell(i)}$.

To define asymptotic freeness, we need to introduce  bundle elements in $\langle \cA^N\rangle$ which are consistent over all $N$. More formally, for integer $q \geq 1$, we denote by $\mathbb C [\cM_q(\cI)]$, the set of finite formal sums $\mathfrak p = \sum_{k} w_k \MM_k$, where $w_k \in \dC$ and $\MM_k \in \cM_q(\cI)$.  If $\mathfrak p =  \sum_{k} w_k \MM_k \in \mathbb C [\cM_q(\cI)]$, then $\mathfrak p (\cA_N) :=  \sum_{k} w_k \MM_k(\cA_N) \in \cE^N_q$ is the associated element in the $\cM$-bundle $\cA^N$.  Now, let $\cJ = \sqcup_{c} \cJ_c$ be a disjoint union of finite sets and let $\ekk : \cJ \to \dN$ be some order function. We may then speak of  $\cJ$-maps (as for $\cI$-maps the dependency in $\ekk$ is implicit). We consider a sequence $\mathfrak p = (\mathfrak p_j)_{j \in \cJ}$, where, for $j \in \cJ_c$, $\mathfrak p_{j} \in \mathbb C [\cM_{\ekk(j)}(\cI_c)]$. We observe that $\mathfrak p_{j}  (\cA^N) \in \langle \cA^N_c \rangle \cap \cE^N_{\ekk(j)}$ for $j \in \cJ_c$ and  set $\mathfrak p (\cA_N) = (\mathfrak p_j(\cA_N))_{j \in \cJ})$.

We then say that  $(\cA^N_c)_{c\in \cC}$ {\em is asymptotically free} if for all triples $(\cJ,\ekk,\mathfrak p)$ as above, all $\cJ$-maps $\MM \in \cM_0(\cJ)$ satisfying conditions (P1) and (P2) in the definition of freeness, we have 
\begin{equation}\label{eq:asfree}
\lim_{N \to \infty} \MM( \mathfrak p(\cA_N)) = 0.
\end{equation}

If the elements in $\cA^N$ are random variables, we speak of convergence in distribution (asymptotic freeness resp.) {\em in probability} or {\em in expectation} if Equation \eqref{eq:cvdist} (Equation \eqref{eq:asfree} resp.) holds in probability or in expectation for all relevant maps $\MM$.

\section{Free cumulants}

\label{sec:cumul}

In this section, we define a notion of free cumulants associated to our notion of freeness. These free cumulants are natural extensions of the free cumulants in free probability theory and they are closely related to the free cumulants introduced in  \cite{kunisky2024tensor} for tensors of fixed dimension $N$. We refer also to \cite{collins2025freecumulantsfreenessunitarily} for a closely related and more complete treatment. This section is independent of the remainder of the paper and it is not used in the proofs of our main results stated in introduction.

\subsection{Definition of free cumulants}

We come back to the general setup of a complex vector space $\cE_1$ and an action from maps $\cM$ on $\cE = \sqcup \cE_p$ as defined in Subsection \ref{susbec:action}. We assume that we have orthonormal identities as per Definition \ref{def:normal}.

Let $\MM_0$ be minimal. We consider the poset $\cP_{\MM_0}$ of elements $\MM \in \cM_0$ such that $\MM_0 \leq \MM$. As is it usual in free probability theory, the free cumulants are routinely defined through Moebius inversion. It is the content of the next lemma. To make sense of the next statement, recall that $\MM' \leq \MM$ implies that $\cE_{\MM'} = \cE_{\MM}$. 

\begin{lemma}[Existence of free cumulants]\label{le:defcumulant}
For all $\MM \in \cP_{\MM_0}$, there exists a unique application $\kappa^{(\MM_0)}_{\MM} : \cE_{\MM} \to \dC$ satisfying the class invariance (CI), morphism (M) and multi-linearity (L) properties such that for all $x \in \cE_{\MM}$
$$
\MM(x) = \sum_{\MM_0 \leq \NN \leq \MM} \kappa_{\NN}^{(\MM_0)} (x).
$$
\end{lemma}

We call the application $\kappa_{\MM}^{(\MM_0)}$ the {\em free cumulant} of $\MM$ in $\cP_{\MM_0}$.

\begin{proof}[Proof of Lemma \ref{le:defcumulant}]
The poset $\cP_{\MM_0}$ is locally finite, that is for any $\MM$, the set of maps $\NN$ such that $\MM_0 \leq \NN \leq \MM$ is finite. We may thus define the Moebius function of $\cP_{\MM_0}$ by the formulas: for all $\NN \leq \MM$,  $\mu ( \MM, \MM) = 1$ and $\mu(\NN,\MM) = - \sum_{\NN \leq \MM' < \MM} \mu( \NN,\MM')$, see \cite[Section 3.7]{zbMATH06016068}. For each $x \in \cE_{\MM_0}$, the function 
$$
\kappa_{\MM}^{(\MM_0)}(x) = \sum_{\MM_0 \leq \NN \leq \MM} \mu(\NN,\MM) \NN(x)
$$
is the unique application $f : \cP_{\MM_0}  \to \dC$ such that for all $\MM \in \cP_{\MM_0}$, $\MM(x) = \sum_{\NN \leq \MM}  f(\NN)$, see \cite[Proposition 3.7.1]{zbMATH06016068}. The properties (CI), (M) and (L) are immediate to check (Property (M) follows from the observation that if $\MM$ has connected components $(\MM_1,\ldots,\MM_{\gamma})$ then $\MM_0 \leq \NN \leq \MM$ can be uniquely decomposed into components  $(\NN_1,\ldots,\NN_{\gamma})$ such that $\NN_i \leq \MM_i$ for all $i$). 
\end{proof}

\par There is a weak form of the substitution property (S) for the free cumulants. 
This weak substitution property is only non-trivial for tensors with more than $2$ edges. Indeed, any $\MM \in \cP_{\MM_0}$ can be written as 
$$
\MM = \MM_{\BB} \circ (\MM_1,\ldots, \MM_k), 
$$
where $\MM_j^{\CL}, 1 \leq j \leq k$ are the connected components of $\MM_0$. We then have the following lemma:
\begin{lemma}[Weak property (S) for the free cumulants]\label{le:subcumulant}
Let $\MM \in \cP_{\MM_0}$ written as above as $\MM = \MM_{\BB} \circ (\MM_1,\ldots,\MM_k)$ where $\MM_j^{\CL}$ are the connected components of $\MM_0$. For all $x \in \cE_{\MM}$, 
\begin{equation*}\tag{S}
    \kappa_{\MM}^{(\MM_0)}(x) =  \kappa_{\MM_\BB}^{(\MM_0)}(\MM_1(x),\ldots, \MM_k(x)).
\end{equation*} 
\end{lemma}

\begin{proof} By construction, any $\MM_0 \leq \MM' \leq \MM$ can be written as $\MM' = \MM'_{\BB} \circ (\MM_1,\ldots, \MM_k)$ with $\MM'_{\BB} \leq \MM_{\BB}$. From the minimality of $\MM_0$, we deduce
$$
\kappa_{\MM}^{(\MM_0)}(x)=\kappa^{(\MM_0)}_{\MM_{\BB} \circ (\MM_1,\ldots,\MM_k)}(x) 
    = \sum_{\NN \leq \MM_{\BB}} [\NN \circ (\MM_1 ,\ldots,\MM_k)]  (x) \mu(\NN,\MM).
$$
From the property (S), $[\NN \circ (\MM_1 ,\ldots,\MM_k)] (x) = \NN ( \MM_1(x), \ldots, \MM_k(x))$.
This concludes the proof. 
\end{proof}

\paragraph{Minimal map.} Observe that if $\MM < \MM'$ differ by a single switch, then the two edges involved in the switch form a $2$-edge cut of $\MM'$. More precisely, the switches that disconnect $\MM'$ are exactly those performed on two edges that are either both bridges (i.e. $1$-edge cuts) or that together form a $2$-edge cut in which neither edge is itself a bridge.

In particular, a map $\MM$ is minimal if each of its connected components is either $3$-edge-connected, or has exactly one bridge and no $2$-edge cut that does not involve this bridge.

\begin{lemma}\label{le:evenmin}
    If a map $\MM$ has only vertices of even degrees, then there exists a unique map $\underline{\MM}$ minimal such that $\underline{\MM} \leq \MM$.
\end{lemma}

\begin{proof} 
    Since every vertex has even degree, each connected component admits an Eulerian circuit. In particular, no edge is a bridge. We may partition the vertices of the map into $3$-edge connected components. Hence, the connected components of the minimal map are in one-to-one correspondence with these $3$-edge connected components.
\end{proof}

\begin{remark}
    In the case of maps with all vertices of even degree, which will be the setup of the following subsection, by Lemma \ref{le:evenmin}, there is only one minimal map smaller than a given map $\MM$. The free cumulant of $\MM$ are then uniquely determined and can be written as a sum over all smaller maps,
$$
\kappa_{\MM}(x) = \sum_{\NN \leq \MM} \mu(\NN,\MM) \NN(x)
$$
\end{remark}

\subsection{Characterization of freeness with cumulants}

In this subsection, we have a finite disjoint collection $(\cA_c)_{c \in \cC}$ where $\cA_c $ is a set of elements in $\cup_{p} \cE_{2p}$. We set $\cA = \sqcup \cA_c$. We say that a subset of elements in $\cE$ is {\em even} if it is a subset of $\sqcup_{t} \cE_{2t}$.  
The main result of this subsection is an analog of Speicher's free cumulant Theorem,  see \cite[Section 5.3]{MR2760897} or \cite[Section 2]{MR2266879}.

\begin{theorem}[Freeness and free cumulants]\label{th:freekappa}
The even families $(\cA_c)_{c \in \cC}$ are free if and only if for all $\langle \cA \rangle $-map $(\MM,x)$ 
connected non-monochromatic, we have $\kappa_{\MM}(x)=0$. 
\end{theorem}

One implication (from right to left) of the theorem is rather straightforward; the other direction will require more arguments.

\paragraph{Join and meet. } We fix $\MM = (\pi, \alpha) \in \cM_0$ with even degree. Let $\MM_0 \leq \MM $ be  minimal (it is unique by Lemma \ref{le:evenmin}) and let $\cP_{\MM_0}$ be the poset of elements $\NN \in \cM_0$ such that $\MM_0 \leq \NN$.  

\begin{lemma}\label{le:joinandmeet}
For two maps $\MM_1, \MM_2 \in \cP_{\MM_0}$ with $\MM_i \leq \MM$, $i = 1,2$, there exists 
\begin{itemize}
    \item[(1)] {\em (join)} a unique smallest map $\MM_1 \vee \MM_2 \in \cP_{\MM_0}$ such that $\MM_1\leq \MM_1 \vee \MM_2$ and $\MM_2\leq \MM_1 \vee \MM_2$, and
    \item[(2)] {\em (meet)} a unique largest map $\MM_1 \wedge \MM_2 \in \cP_{\MM_0}$ such that $\MM_1 \wedge \MM_2 \leq \MM_1$ and $\MM_1 \wedge \MM_2 \leq \MM_2$.
\end{itemize}
\end{lemma}
\begin{proof}
(2) follows from the claim that two vertices are in the same connected component in $\MM_1 \wedge \MM_2$ if and only if they are in the same component in $\MM_1$ and in $\MM_2$. As in the proof of Lemma \ref{le:evenmin}, we may then use that disconnecting switches can be done in any order (for maps with even degrees).
\newline 
As for (1), since $\MM_1$ and $\MM_2$ are both smaller than $\MM$, we have
$$
\MM_1 \vee \MM_2 = \bigwedge \{ \NN \in \cP_{\MM_0} \vert \MM_1\leq \NN, \MM_2\leq \NN \}.
$$
The conclusion follows.
\end{proof}

Now, we further fix $x \in \cE_{\MM}$ with $x_v \in \langle \cA \rangle$ for all $v \in V(\MM)$ (in other words, $(\MM,x)$ is an $\langle \cA \rangle$-map). Let $\widehat \MM_1,\widehat \MM_2 \in \cP_{\MM_0}$ with $\widehat \MM_i \leq \MM$, $i = 1,2$.

\begin{lemma}\label{le:minchapeau}
If $\widehat \MM_1$ and $\widehat \MM_2$ both satisfy (P1)-(P2), then $\widehat \MM_1 \wedge \widehat \MM_2$ also satisfies (P1)-(P2).
\end{lemma}
\begin{proof}
We can restrict to the case where $\widehat \MM_1$ and $\widehat \MM_2$ have only monochromatic connected components because if there is a minimal non-monochromatic connected component in one of them, it will necessarily be also a connected component of $\MM_0$ and therefore of $\widehat \MM_2$ and $\widehat \MM_1 \wedge \widehat \MM_2$.
    
Firstly, it is immediate that $\widehat \MM_1 \wedge \widehat \MM_2$ satisfies (P1). Now fix an Eulerian circuit of $\MM$. Along this circuit, group consecutive vertices belonging to the same family into maximal monochromatic intervals, which we denote cyclically by $(C_1,C_2,\ldots,C_n)$. By maximality, $C_j$ and $C_{j+1}$ are never of the same color, where the indices are understood cyclically. By property (P2) we know that all vertices appearing in a same interval $C_j$ are in the same connected component of $\widehat \MM_1$ and of $\widehat \MM_2$. For $r\in\{1,2\}$, let $\PP_r$ be the partition of $(C_1,\ldots,C_n)$ defined by
$$ C_i\sim_{\PP_r}C_j \quad\Longleftrightarrow\quad C_i\text{ and }C_j \text{ belong to the same connected component of }\widehat\MM_r.$$
Condition (P2) implies that $\PP_1$ and $\PP_2$ are non-crossing partitions of $(C_1,\ldots,C_n)$, since if there is a crossing between two blocks, these two blocks must belong to the same $3$-edge component.  
Moreover, since the connected components of $\widehat\MM_r$ are monochromatic, every block of $\PP_r$ consists of intervals $C_j$ having the same color.

The meet $\widehat\MM_1\wedge\widehat\MM_2$ corresponds to the partition
$$\PP:=\PP_1\wedge\PP_2.$$
Equivalently, the blocks of $\PP$ are the non-empty intersections of a block of $\PP_1$ with a block of $\PP_2$. The meet of two non-crossing partitions is non-crossing. It remains to verify (P2). We describe the procedure to obtain this map from $\MM$ with only chromatic disconnecting switches. Moreover, this sequence may be chosen so that no proper intermediate map satisfies (P1).

We proceed by induction on $n$, the number of blocks of $\PP$. The assertion is immediate when $n=1$. Assume that $n\geq 2$. Every non-crossing partition contains an interval block. Since consecutive intervals $C_j$ have different colors whereas every block of $\PP$ is monochromatic, such an interval block must be a singleton. After a cyclic relabelling, we may assume that $\{C_1\}\in\PP$. Let
$$ e_{n,1} \quad\text{and}\quad e_{1,2} $$
be the two edges followed by the Eulerian circuit when it enters and leaves $C_1$. 
By maximality of $C_1$, we have that $C_n$ and $C_2$ are not of the same color as $C_1$. Let 
$$i := \mathrm{max}_{1\leq j \leq n} \{ j : C_j \overset{\PP}{\sim} C_2 \}$$
the largest index $j$ such that $C_j$ is in the block of $C_2$. If $i=n$ then switch the edge visited by the Eulerian circuit between $C_n$ and $C_1$ say $e_{n,1}$ with the one visited between $C_1$ and $C_2$ say $e_{1,2}$. This isolates $C_1$ and gives another non-crossing partition where we can apply the recurrence (as $C_n \overset{\PP}{\sim} C_2$). Otherwise $i<n$ and then switch the edge $e_{n,1}$ visited by the Eulerian circuit between $C_n$ and $C_1$ with $e_{i,i+1}$ visited between $C_i$ and $C_{i+1}$. This switch is chromatic. It is disconnecting because the switches $e_{n,1}$-$e_{1,2}$ and $e_{n,2}$-$e_{i,i+1}$ give three connected components and it would be at most two if the map is still connected after the switch $e_{n,1}$-$e_{i,i+1}$. We can then apply the recurrence in the two connected components. Indeed, by the non-crossing property of $\PP$, no block of $\PP$ meets both cyclic intervals $(C_1,\ldots,C_i)$ and $(C_{i+1},\ldots,C_n)$ since such a block would cross the block $B$, which contains both $C_2$ and $C_i$.

During the construction, distinct blocks of $\PP$ having the same color are never joined by a new (non-chromatic) edge: in the first case such an edge is created only when the two intervals belong to the same block, and in the second case both new edges are chromatic. Hence, unless the target partition $\PP$ has already been reached, at least one connected component contains several target blocks and is non-monochromatic. The recursive construction gives a further disconnecting chromatic switch inside this component, so it is not minimal. Thus no proper intermediate map satisfies (P1).
This completes the proof.

\end{proof}

\begin{proof}[Proof of Theorem \ref{th:freekappa} : proof of $\boxed{\Leftarrow}$]

Assume that $\kappa_{\MM}(x)=0$ for all $\langle \cA \rangle $-map $(\MM,x)$ 
connected non-monochromatic. Take an $\cA$-map $(\MM,x)$ 
such that all the associated $(\widehat \MM,x)$ satisfying (P1)-(P2) have a centered connected component. We denote these maps $\widehat \MM _1$,\ldots,$\widehat \MM _d$. On the one hand, if $\NN$ does not satisfy (P1) (i.e. there is no $\widehat \MM$ such that $\NN \leq \widehat \MM$), then $\kappa_{\NN}(x)=0$. Hence, we get
$$
\MM(x) = \sum_{\exists i : \NN \leq \widehat \MM_i } \kappa_{\NN} (x).
$$
On the other hand, we can show using Lemma \ref{le:minchapeau} that this remaining sum is a linear combination of the $\widehat \MM (x)$ which are all zero. We have by inclusion-exclusion principle,
\begin{align*}
    \MM(x) &= \sum_{\NN \leq \MM } \mathbf{1}_{\{ \exists i : \NN \leq \widehat \MM_i \}} \kappa_{\NN} (x) \\
    &= \sum_{\NN \leq \MM } \kappa_{\NN} (x) \sum_{l=1}^d (-1)^{l-1} \sum_{I \subset \{1,\cdots,d\} \atop \vert I \vert=l } \prod_{j \in I} \mathbf{1}_{\{ \NN \leq \widehat \MM_j \}} .
\end{align*}
Hence, we obtain
\begin{align*}
    \MM(x) = \sum_{l=1}^d (-1)^{l-1} \sum_{I \subset \{1,\cdots,d\} \atop \vert I \vert=l }  \underbrace{\sum_{\NN \leq \bigwedge_{j \in I} \widehat \MM_j } \kappa_{\NN} (x)}_{=(\bigwedge_{j \in I} \widehat \MM_j) (x)} .
\end{align*}
By Lemma \ref{le:minchapeau}, $\bigwedge_{j \in I} \widehat \MM_j$ satisfies (P1)-(P2), so it is centered. This gives $\MM(x)=0$ and thus the families $(\cA_c)$ are free.
\end{proof}

In order to prove the converse statement, we first prove the following Lemma \ref{le:equacumu} and Lemma \ref{le:centeredkappa}, which extend results known in free probability. We will adapt in our setting the proofs given in \cite[Section 5.3]{MR2760897} and \cite[Section 2]{MR2266879}.

\par Now, fix integers $1 \leq r < k$ and a partition $p$ of $\{1,\ldots ,k\}$ into $r$ blocks.  If $\underline{\MM}^k$ is a minimal map and $\MM^k = (\pi^k,\alpha)$ is a map in $\cP_{\underline{\MM}^k}$ with $k$ vertices, we can define a map $\MM^r = \MM^k / p = (\pi^r,\alpha)$ with $r$ vertices by concatenation of the cycles of $\pi$ in the same blocks of the partition $p$. More precisely, if $C_1 = (e_1,\ldots,e_{l_1})$, $C_2 = (e_{l_1+1},\ldots,e_{l_1 + l_2})$, \ldots, $ C_c = (e_{l_1 + \cdots + l_{c-1}+1},\ldots ,e_{l_1 + \cdots + l_c}) $ are the cycles of $\pi^k$ associated with a block of $p$ with $c$ elements ordered so that with $e_1 < e_{l_1+1} < \ldots$ and $e_{l_j +1} = \min (C_j)$, they are merged as a single cycle $(e_1,\ldots, e_{l_1 + \cdots + l_c})$ in $\pi^r$. We  also set $\underline{\MM}^r = \underline{\MM}^k / p$.

Now assume there is a map ${\tilde \MM} \in \cP_{\MM_0}$ such that ${\tilde \MM} = \MM^r \circ(\mathfrak g_1, \ldots, \mathfrak g_r) = \MM^k \circ(\MM_1, \ldots, \MM_k) $ where $\MM^k$ and $\MM^r = \MM^k/p$ are as above and $\MM_{j}$, $\mathfrak{g_j}$ are maps with boundaries of the proper size. For each $\underline{\MM}^r \leq \MM \leq \MM^r$ (with $r$ vertices), $\MM = (\pi^r , \beta)$, there is a unique map $\MM^* = ( \pi^k , \beta^*)$ satisfying $\underline{\MM}^k \leq \MM^* \leq \MM^k$ (with $k$ vertices) such that $\MM \circ(\mathfrak g_1, \ldots, \mathfrak g_r) = \MM^* \circ(\MM_1, \ldots, \MM_k) $. It is given by $\beta^*=\theta \beta \theta$ where $\theta$ is the involution given by $\theta_{\vert C_j} = \theta_j$ where $\mathfrak g_j = (\pi^k_{\vert C_j}, \theta_j)$.

The application $\cdot^* : \cP_{\underline{\MM}^r} \to \cP_{(\underline{\MM}^r)^*}$ is an embedding of partially ordered sets which preserves the partial order. Moreover, for every
$\underline{\MM}^r \leq \NN \leq \MM$, its restriction induces an isomorphism
of intervals $[\NN,\MM]_{\cP_{\underline{\MM}^r}}\simeq [\NN^*,\MM^*]_{\cP_{(\underline{\MM}^r)^*}}$, by the definition of $\beta^*$. Since the Möbius function of a locally finite poset depends only on the
isomorphism type of the corresponding interval, see for instance \cite[Proposition 3.6.2]{zbMATH06016068}, we obtain
\begin{equation}\label{eq:mustar}
    \mu(\NN^*,\MM^*)=\mu(\NN,\MM).
\end{equation}
An important remark is that we have $(\MM^r)^*=\MM^k$ but in general $(\underline{\MM}^r)^*\neq \underline{\MM}^k$. The following detailed example shows it.

\begin{example}
    We take ${\tilde \MM} = \MM^4 = (\pi^4, \alpha)$ with 
$$ \pi^4 = (1,2) (3,4,5,6) (7, 8)(9, 10)$$
$$ \alpha = (1, 3)(2, 7)(4, 8)(5, 9)(6, 10)$$
and then $\MM_1=\ST_2$, $\MM_2=\ST_4$, $\MM_3=\ST_2$, $\MM_4=\ST_2$ all star maps. The maps $\MM^4$ and $$\underline{\MM}^4=(\pi^4,\beta=(1, 2)(3, 4)(5, 6)(7, 8)(9, 10))$$ are given in Figure \ref{fig:map_mk}.
Then we take the partition $p=\{ \{1,2\}, \{3\}, \{4\} \}$ and $\mathfrak g_1 = ((1, 2)(3, 4, 5, 6), \theta_1=(1, 3)(2)(4)(5)(6))$ (a map $\ST_2$ and a map $\ST_4$ linked by one of their edges), $\mathfrak g_2= \MM_3 = \ST_2$, $\mathfrak g_3 = \MM_4 = \ST_2$. Then we have
$$ \pi^3 = (1, 2, 3, 4, 5, 6) (7, 8)(9, 10)$$
and hence $\MM^3=(\pi^3,\alpha)$, $\underline{\MM}^3=(\pi^3,\beta)$ are given in Figure \ref{fig:map_mr}.
\begin{figure}[h!]
\begin{minipage}[c]{.46\linewidth} 
      \begin{tikzpicture}[scale=0.5]
    \filldraw[gray] (1.5,0) circle (3pt);
    \filldraw[gray] (0,-2) circle (3pt);
    \filldraw[gray] (3,-2) circle (3pt);
    \filldraw[gray] (3.5,0) circle (3pt);
    \draw (1.5,0) -- (0,-2);
    \draw (1.5,0) .. controls (2,0.5) and (3,0.5) .. (3.5,0);
    \draw (1.5,0) .. controls (2,-0.5) and (3,-0.5) .. (3.5,0);
    \draw (0,-2) -- (3,-2) ;
    \draw (1.5,0) -- (3,-2) ;

    \filldraw[gray] (7.5,0) circle (3pt);
    \filldraw[gray] (6,-2) circle (3pt);
    \filldraw[gray] (9,-2) circle (3pt);
    \filldraw[gray] (9.5,0) circle (3pt);
    \draw (7.5,0) .. controls (6.5,1.5) and (8.5,1.5) .. (7.5,0);
    \draw (7.5,0) .. controls (6.5,-1.5) and (8.5,-1.5) .. (7.5,0);
    \draw (6,-2) .. controls (5,-0.5) and (7,-0.5) .. (6,-2);
    \draw (9,-2) .. controls (8,-0.5) and (10,-0.5) .. (9,-2);
    \draw (9.5,0) .. controls (8.5,1.5) and (10.5,1.5) .. (9.5,0);
\end{tikzpicture}
\caption{Maps $\MM^4$ and $\underline{\MM}^4$.}
\label{fig:map_mk}
   \end{minipage} \hfill 
   \begin{minipage}[c]{.46\linewidth} 
      \begin{tikzpicture}[scale=0.5]
    \filldraw[gray] (1.5,0) circle (3pt);
    \filldraw[gray] (3,-2) circle (3pt);
    \filldraw[gray] (3.5,0) circle (3pt);
    \draw (1.5,0) .. controls (0,1) and (0,-1) .. (1.5,0);
    \draw (1.5,0) .. controls (2,0.5) and (3,0.5) .. (3.5,0);
    \draw (1.5,0) .. controls (2,-0.5) and (3,-0.5) .. (3.5,0);
    \draw (1.5,0) .. controls (1.5,-1) and (2,-2) .. (3,-2) ;
    \draw (1.5,0) -- (3,-2) ;

    \filldraw[gray] (7.5,0) circle (3pt);
    \filldraw[gray] (9,-2) circle (3pt);
    \filldraw[gray] (9.5,0) circle (3pt);
    \draw (7.5,0) .. controls (6.5,1.5) and (8.5,1.5) .. (7.5,0);
    \draw (7.5,0) .. controls (6.5,-1.5) and (8.5,-1.5) .. (7.5,0);
    \draw (7.5,0) .. controls (6,1) and (6,-1) .. (7.5,0);
    \draw (9,-2) .. controls (8,-0.5) and (10,-0.5) .. (9,-2);
    \draw (9.5,0) .. controls (8.5,1.5) and (10.5,1.5) .. (9.5,0);
\end{tikzpicture}
\caption{Maps $\MM^3$ and $\underline{\MM}^3$.} \label{fig:map_mr} 
   \end{minipage}
\end{figure}
\par Finally, we can compute here
$$ (\underline{\MM}^3)^*= (\pi^4, \theta \beta \theta = (1, 4)(2, 3)(5, 6)(7, 8)(9, 10)). $$
\begin{figure}[h!]
    \centering
    \begin{tikzpicture}[scale=0.5]
    \filldraw[gray] (7.5,0) circle (3pt);
    \filldraw[gray] (6,-2) circle (3pt);
    \filldraw[gray] (9,-2) circle (3pt);
    \filldraw[gray] (9.5,0) circle (3pt);
    \draw (7.5,0) .. controls (6.5,1.5) and (8.5,1.5) .. (7.5,0);
    \draw (7.5,0) .. controls (6.5,0) and (6,-1) .. (6,-2);
    \draw (7.5,0) .. controls (7.5,-1) and (7,-2) .. (6,-2);
    \draw (9,-2) .. controls (8,-0.5) and (10,-0.5) .. (9,-2);
    \draw (9.5,0) .. controls (8.5,1.5) and (10.5,1.5) .. (9.5,0);
\end{tikzpicture}
\caption{Map $(\underline{\MM}^3)^*$.} \label{fig:map_mrstar}
\end{figure}
We find $(\underline{\MM}^3)^*\neq \underline{\MM}^4$.
\end{example}

The next lemma asserts that the application $\cdot^*$ allows to express $\kappa_\MM$ with $\MM \in  \cP_{\underline{\MM}^r} $ in terms of linear combinations of $\kappa_{\MM'}$ with $\MM' \in \cP_{(\underline{\MM}^r)^*} $. 

\begin{lemma}\label{le:equacumu}
Let ${\tilde \MM} \in \cP_{\MM_0}$ such that ${\tilde \MM} = \MM^r \circ(\mathfrak g_1, \ldots, \mathfrak g_r) = \MM^k \circ(\MM_1, \ldots, \MM_k) $ where $\MM^k$ and $\MM^r = \MM^k/p$ are as above and $\MM_{j}$, $\mathfrak{g_j}$ are maps with boundaries of the proper size. Then for all $\MM\leq \MM^r$ and $x \in \cE_{\MM_0}$, we have 
$$
\kappa_{\MM}(\mathfrak{g}_1(x),\ldots,\mathfrak{g}_r(x))= \sum_{(\underline{\MM}^r)^* \leq \MM'\leq \MM^* \atop \MM' \vee (\underline{\MM}^r)^* =\MM^* }  \kappa_{\MM'}(\MM_1(x),\ldots,\MM_k(x)).
$$
\end{lemma}

\begin{proof} 
    Let $\MM\leq \MM^r$. First, by Equation \eqref{eq:mustar}
\begin{align*}
    \kappa_{\MM}(\mathfrak{g}_1(x),\ldots,\mathfrak{g}_r(x)) &= \sum_{\underline{\MM}^r \leq \MM'\leq \MM} \MM'(\mathfrak{g}_1(x),\ldots,\mathfrak{g}_r(x)) \mu(\MM',\MM) \\
    &= \sum_{\underline{\MM}^r \leq \MM' \leq \MM} \MM'(\mathfrak{g}_1(x),\ldots,\mathfrak{g}_r(x)) \mu(\MM'^*,\MM^*) ,
\end{align*}
and this gives
\begin{align*}
    \kappa_{\MM}(\mathfrak{g}_1(x),\ldots,\mathfrak{g}_r(x))
    &= \sum_{(\underline{\MM}^r)^* \leq \NN \leq \MM^* } \NN(\MM_1(x),\ldots,\MM_k(x)) \mu(\NN,\MM^*)  \\
    &= \sum_{(\underline{\MM}^r)^* \leq \NN \leq \MM^* } \sum_{\MM'\leq \NN} \kappa_{\MM'}(\MM_1(x),\ldots,\MM_k(x)) \mu(\NN,\MM^*) .
\end{align*}
Then exchanging the two sums, we get
\begin{align*}
    \kappa_{\MM}(\mathfrak{g}_1(x),\ldots,\mathfrak{g}_r(x))
    &= \sum_{\underline{\MM}^k \leq \MM'\leq \MM^*} \kappa_{\MM'}(\MM_1(x),\ldots,\MM_k(x)) \underbrace{\sum_{(\underline{\MM}^r)^* \vee \MM' \leq \NN \leq \MM^* } \mu(\NN,\MM^*)}_{=\delta_{(\underline{\MM}^r)^* \vee \MM'=\MM^*}}  \\
    &= \sum_{\underline{\MM}^k \leq \MM'\leq \MM^* \atop \MM' \vee (\underline{\MM}^r)^* =\MM^* }  \kappa_{\MM'}(\MM_1(x),\ldots,\MM_k(x)).
\end{align*}
This achieves the proof.
\end{proof}

The next lemma asserts that $\kappa_{\MM}(x) = 0$ when a coordinate of $x$ is equal to the identity as soon as a minimal map of $\MM$ isolates this vertex.

\begin{lemma}\label{le:centeredkappa}
    Let $\MM$ be a connected map in $\cP_{\MM_0}$ with at least two vertices. If, for some integer $t \geq 1$, one connected component of $\MM_0$ is a bouquet map $\BB^\tau_{2t}$ at vertex $v$ for some pairing $\tau$, then for any $x \in \cE_{\MM}$ such that $x_v = 1^\tau_{2t}$, we have  $\kappa_{\MM}(x)=0$.
\end{lemma}

\begin{proof}
Let $k = |V(\MM)| \geq 2$.
    We prove the result by induction on $k\geq 2$. In the initial case $k=2$, since $\MM \ne \MM_0$, we have $\MM =\MM_{\BB}\circ (\MM_{1},\MM_2)$, with $\MM_j \in \cM_2$, $\MM_1^{\CL}=\BB^\tau_{2t}$ for some $t\geq 1$ is attached to vertex $v$ and $\MM_\BB \in \cM_0$ is a melon of order $2$. Then by (S)-(Id),
    $$
    \kappa_{\MM}(x)=\MM(x)-\MM^{\CL}_1(1^\tau_{2t}) \MM^{\CL}_2(x)= \MM^{\CL}_2(x)- 1.\MM^{\CL}_2(x) = 0.
    $$

    Now assume the statement is true for all maps with $2 \leq l<k$ vertices. By Lemma \ref{le:defcumulant}, we have
    $$
    \MM(x) = \kappa_{\MM}(x) + \sum_{\MM_0 \leq \NN < \MM } \kappa_{\NN}(x).
    $$
    From the induction hypothesis $\kappa_{\NN}(x) = 0$ for all $\NN$ such that $v$ is not isolated in $\NN$. If $v$ is isolated in $\NN \leq \MM$, then denote $\NN'$ the restriction of the map $\NN$ to $V(\MM) \backslash \{v\}$ and $x' = (x_u)_{u \ne v}$. 
    We get from property (M), 
    $ \kappa_{\NN}(x) = \kappa_{\BB^\tau_{2t}}(1^\tau_{2t})\kappa_{\NN'}(x') = \BB^\tau_{2t}(1^\tau_{2t}) \kappa_{\NN'}(x') =  \kappa_{\NN'}(x')$. 
    Also, denote $\MM'$ the map on $V(\MM)\backslash \{v\}$ obtained by removing vertex $v$ and matching the corresponding boundary edges according to $\tau$. 
    Necessarily, $\MM'$ has $\gamma(\MM)+t-1$ connected components, since $\MM_0$ contains the bouquet map, and by definition of the poset.
    Then, we have from property (Id), $\MM(x) = \MM'(x')$.
    We deduce from Lemma \ref{le:defcumulant}, 
$$
\MM'(x') = \MM(x) =  \kappa_{\MM}(x) + \sum_{\MM_0 \leq \NN' \leq \MM'}\kappa_{\NN}(x) = \kappa_{\MM}(x) + \MM'(x').
$$ 
Hence $\kappa_{\MM}(x)=0$ as desired. 
\end{proof}

\begin{proof}[Proof of Theorem \ref{th:freekappa} : proof of \boxed{\Rightarrow}] We assume that the families $(\cA_c)_{c \in \cC}$ are free. Let $(\MM,x)$ be a connected non-monochromatic $\langle \cA \rangle $-map. We  write $\MM = \MM_{\BB} \circ (\MM_1,\ldots,\MM_k)$ where $\MM_{\BB} \in \cM_{0}$, $\MM_j \in \cM_{2t_j}$ and $\MM_j^{\CL}$ are the connected components of the minimal map $\MM_0$. For each $j \in J_1 \subset \INT{k}$, $(\MM^{\CL}_j,x) $ is minimal monochromatic, so that $\MM_j(x) \in \langle \cA_{c_j}\rangle$ for some $c_j$. For each $j \in J_2 = \INT{k} \backslash J_1$, $(\MM_j^{\CL},x)$ is minimal non-monochromatic. From the (S) property in Lemma \ref{le:subcumulant}, we have 
$$
\kappa_{\MM}(x) =  \kappa_{\MM_{\BB}}( \MM_1(x),\ldots, \MM_k(x)) = \kappa_{{\tilde \MM}} (x'),
$$
where ${\tilde \MM} = \MM_{\BB} \circ (\NN_1 , \ldots, \NN_k) \in \cP_{\NN_0}$ where $\NN_j = \ST_{2t_j}$ for $j \in J_1$, $\NN_j = \MM_j$ for $j \in J_2$, $\NN_0$ is the minimal map with connected components $\NN_j^{\CL}$, $x'_j = \MM_j (x)$ for a vertex $j \in J_1$ and $x'_v = x_v$ for a vertex $v \in V_{\MM_j}$, $j \in J_2$.
Next, using Lemma \ref{le:centeredkappa} and property (L), we deduce
$$
\kappa_{\MM}(x) = \kappa_{{\tilde \MM}}(y),
$$
where for $j \in J_1$,  $y_j = \MM_j(x) - \MM_j^{\CL} (x) 1_{2t_j}  \in \cE_{2t_j}\cap \langle \cA_{c_j} \rangle $ while for $j \in J_2$, $v \in V_{\MM_j}$, $y_v = x_v$.

To prove that $\kappa_{{\tilde \MM}}(y) = 0$, we start with a first case on $y$: we assume that there is a non-monochromatic connected component in $\NN_0$. Then we write 
$$
\kappa_{{\tilde \MM}}(y) =  \sum_{\NN_0 \leq \NN \leq {\tilde \MM}} \mu( \NN,{\tilde \MM}) \NN(y). 
$$
By definition of freeness, in the above sum we have $\NN (y)  = 0$ since in this first case, we have that $\widehat \NN$ must have a minimal non-monochromatic connected component. Hence $\kappa_{{\tilde \MM}}(y) = 0$ in this case.

For the general case, the proof is by induction on $k$. We assume that $\kappa_{\MM } (x) = \kappa_{{\tilde \MM}} (y) = 0$ for all connected non-monochromatic $(\MM,x) \in \cP_{\MM_0}(\langle \cA \rangle)$ where $\MM_0$ has $k$ connected components. For the initial step $k = 1$, $\MM$ is minimal non-monochromatic and thus $\MM(x) = \kappa_{\MM} (x) = 0$ by the definition of freeness.

For the inductive step, from what precedes, it suffices to consider the complementary of the first case where $\NN_0$ has only monochromatic connected components. If these connected components are all in different families $\cA_c$ then $\widehat{\tilde \MM} = \NN_0$ and $\kappa_{{\tilde \MM}} (y)=0$ as all the maps involved in the sum are zero by freeness. Otherwise, there are two vertices, say $j_1,j_2$ in $J_1$, which are in the same family, $c_{j_1} = c_{j_2}$, and connected by an edge in ${\tilde \MM}$ such that this edge is part of a switch which disconnects ${\tilde \MM}$. For ease of notation, we may assume that $j_1 = 1$ and $j_2 = 2$. We consider the partition $p$ on ${{\tilde \MM}}$ with $|V_{{\tilde \MM}}|-1$ blocks where all vertices are a singleton except $\{1,2\}$ forming a block of size $2$, which we call block $1$. 

We then apply Lemma \ref{le:equacumu} with $k = r +1 = |V_{{\tilde \MM}}|$, $\MM^k = \MM_\BB$, $\MM^r = \NN_{\BB} = \MM_{\BB} / p$, $\MM_j = \NN_j, (\mathfrak g_2,\ldots, \mathfrak{g} _{k-1} )  = (\NN_3,\ldots, \NN_k)$ and $\mathfrak g_{1}$ being two star maps connected by a single edge. We find
$$
\kappa_{\NN_{\BB}} (\mathfrak g_{1}  (y_{1},y_{2} ), y_3, \ldots ) =  \kappa_{{\tilde \MM}} (y) + \sum_{\underline{\MM}_{\BB} \leq \MM' < \MM_{\BB}  \atop \MM' \vee (\underline{\NN}_{\BB})^* =\MM_\BB } \kappa_{\MM'} (y).
$$
The first term on the left-hand side is zero by the recursion hypothesis. We now prove that every term in the sum on the right-hand side vanishes. By definition, $(\underline{\NN}_{\BB})^*$ is the map where all vertices are isolated except $1$ and $2$. Hence any $\MM'$ contributing on the right-hand side is a map with two connected components, $1$ and $2$ being in different connected components. Recall that $1$ and $2$ were chosen in the same family. Since $({\tilde \MM},y)$ is non-monochromatic, at least one of the two connected components of $(\MM' \circ ( \NN_1, \ldots, \NN_k),y)$ is also non-monochromatic, and it has strictly fewer minimal connected components than $\tilde \MM$, since the other component contains one of the vertices $1$ and $2$. By the recursion hypothesis and morphism property (M), it follows that $\kappa_{\MM'}(y) = 0$. Therefore $\kappa_{{\tilde \MM}} (y) = 0$. This concludes the proof.\end{proof}

\subsection{Application to associativity of freeness}

In the previous general section on freeness, we have left aside an important question: if $\cA_{1}$, $\cA_2$ and $\cA_3$ are free, is it true that $\cA_1$ and $\cA_2 \cup \cA_3$ are free? For even families, the free cumulants give an immediate positive answer.

\begin{corollary}[Associativity of freeness]
Let  $(\cA_{1}, \cA_2, \cA_3)$ be free and even families. Then $(\cA_1, \cA_2 \cup \cA_3)$ are free.
\end{corollary}

\begin{proof}
Let $\cA = \cA_1 \sqcup \cA_2 \sqcup \cA_3$, $\cA_{4} = \cA_2 \cup \cA_3$ and $\cA' = \cA_{1} \sqcup \cA_{4}$. By Theorem \ref{th:freekappa}, we should prove that for any $\langle \cA' \rangle$-map $(\MM,x)$ connected non-monochromatic (for the colors $\{ 1, 4\}$), we have $\kappa_{\MM}(x) = 0$. We write any element in $\langle \cA_4 \rangle$ as the sum of two elements in $\cA_2$ and $\cA_3$. Then, by the multi-linearity property of the cumulants, we can develop the cumulant $\kappa_{\MM}(x) $ into a sum of cumulants of $\langle \cA \rangle$-maps $(\MM_j, x_j)$. It remains to apply Theorem \ref{th:freekappa} for $(\cA_{1}, \cA_2, \cA_3)$. 
\end{proof}

\subsection{Application to sum of free elements}

In this subsection, we give the classical illustration of cumulants to prove a free central limit theorem. Similar results appear in \cite{kunisky2024tensor,collins2025freecumulantsfreenessunitarily} in a slightly different setting.

For short notation, if $\MM \in \cM_0$ is a $p$-regular map (all vertices have degree $p$) and $a \in \cE_p$, we set $\MM(a) = \MM(a,\ldots,a)$ and $\kappa_{\MM}(a) = \kappa_{\MM}(a,\ldots,a)$ (all vertices are colored with $a$). We say that $a \in \cE_p$ is {\em centered} if $p$ is odd or, for $p$ even, $\MM(a) = 0$ for all $p$-regular maps $\MM$ with exactly $1$ vertex (that is bouquet maps with permuted matching of directed edges). As above, for a given $p$ and $\sigma \in \mathrm{S}_p $ the melon $\FF^{\sigma}_{p}$ (or Frobenius pair) is the map with $2$ vertices where the two vertices say $v_1,v_2$, are connected by $p$ edges, the $i$-th edge of $v_1$ being connected to the $\sigma(i)$-th edge of $v_2$.  

\begin{theorem}[Free CLT for tensors]
Let $p \geq 2$, $(a_i)_{i \geq 1} \in \cE_p$ be a collection of centered free elements. Assume that for all $p$-regular maps $\MM \in \cM_0$, there exists $C(\MM)$ such that for all $i$: $|\MM(a_i)| \leq C(\MM)$. Assume moreover that for all $\sigma \in \mathrm{S}_p$, $\FF^\sigma _p (a_i) = t^\sigma_p$ is independent of $i$. Then, 
$$
\frac{1}{\sqrt n} \sum_{i=1}^n a_i
$$
converges toward a distribution, say $\SC$, characterized by $\kappa_{\FF_p^\sigma} (\SC) = \FF_p^\sigma (\SC) = t_p ^\sigma$ and $\kappa_{\MM} (\SC) = 0$ for any other connected map $\MM$. 
\end{theorem}

\begin{proof}
Let $s_n =
\frac{1}{\sqrt n} \sum_{i=1}^n a_i
$. It suffices to compute the limit of $\kappa_{\MM}(s_n)$ for all connected $p$-regular maps $\MM \in \cM_0$. Using the multi-linearity and Theorem \ref{th:freekappa}, we have 
$$
\kappa_{\MM} (s_n) = n^{-v/2} \sum_{i=1}^n \kappa_{\MM} (a_i),
$$
where $v$ is the number of vertices of $\MM$. If $v=1$ then $\kappa_{\MM} (s_n) = 0$ by assumption. If $v\geq 3$, denote $\tilde C(\MM):=\mathrm{max}_{\NN \leq \MM}C(\NN)$, then $|\kappa_{\MM} (s_n)|\leq n^{1-v/2} \tilde C(\MM)$ which goes to $0$. Finally if $v=2$, then $\MM = \FF_p ^\sigma$ for some $\sigma$ and $\kappa_{\MM} (s_n) = \MM(s_n) = t_p^\sigma$ by assumption.
\end{proof}

\section{Asymptotic freeness for random and deterministic tensors}
\label{sec:SD}

We take $\cE^N$ as in Subsection \ref{subsec:TR}. Let $\cI_0$ be a finite set and $\ell_0 : \cI_0 \to \{1,2,\ldots \}$. We consider a finite and deterministic collection $\cA^N_0 = (A^N_i)_{i \in \cI_0}$ of elements in $\cE^N$ such that for all $N$, $i$, we have $A^N_i \in \cE^N_{\ell_0(i)}$.  We assume that $\cA^N_0$ is stable by taking entry-wise complex conjugation: for all $i \in I_0$, $\bar A_i^N \in \cA^N_0$. We will prove Theorem \ref{th:free1}, Theorem \ref{th:free1b}, Theorem \ref{th:free2} and Theorem \ref{th:free3} stated in Introduction.

\subsection{Assumptions on $\cA^N_0$}

We start by giving a more formal statement of Assumptions (A1) and (A2).

\begin{enumerate}
    \item[(A1)] \label{A1} For all $\MM \in \cM_0(\cI_0)$, there exists a constant $C(\MM)$ such that for all $N \geq 1$
\begin{equation*}\label{eq:bdA0}
\ABS{\MM ( \cA_0^N)} \leq C(\MM),
\end{equation*}
where we recall $\MM ( \cA_0^N)$ is the corresponding action on $\cE_0 = \dC$. 
\end{enumerate}

To formulate Assumption (A2), we introduce another kind of maps. A {\em combinatorial hyper-map} is a pair $\MM = (\pi,\alpha)$ of permutations in $S_{m}$ for some integer $m\geq 1$. The set $\vec E (\MM) = \INT{m}$ are the directed edges (or half-edges), $\pi$ has $n$ cycles, denoted by $V(\MM)$, ordered by least elements which are the directed edges attached to each vertex and $\alpha$ is a permutation whose cycles have length at least two, denoted by $E(\MM)$. The cycles of $\alpha$ define the hyper-edges of $\MM$. Hence, compared to maps, in an hyper-map the permutation $\alpha$ is not necessarily an involution. We denote by $\widehat \cM_0$ the set of combinatorial hyper-maps, we have $\cM_0 \subset \widehat \cM_0$. The definition of vertex degrees, maps with boundaries, colored maps and action of maps extend verbatim to hyper-maps.  
Similarly, in the central example, $\cE^N_p = (\dC^N)^{\otimes p}$, $\cE^N = \sqcup_p \cE^N_p$ introduced in Subsection \ref{subsec:TR}, the action given by \eqref{eq:tracem} for maps extends verbatim to hyper-maps.

 We extend assumption (A1) to all hyper-maps.
\begin{enumerate}
    \item[(A2)] \label{A2} For all hyper-maps $\MM \in \widehat \cM_0(\cI_0)$, there exists a constant $C(\MM)$ such that  for all $N \geq 1$
\begin{equation*}
\ABS{\MM ( \cA_0^N)} \leq C(\MM).
\end{equation*}
\end{enumerate}


\subsection{Schwinger-Dyson equations for Gaussian symmetric random tensors}

In this subsection, we fix $p \geq 1$.  Let $\cI  =\cI_0 \cup \{ \SC \}$ and $\ell : \cI \to \{1,2,\ldots \}$ such that $\ell = \ell_0$ on $\cI_0$ and $\ell(\SC) = p$. We define $A^N_\SC = W^N$ and $\cA^N = \cA_0^N \cup \{ W^N \} = \{ A_i^N : i \in \cI \}$.

If $\MM  = (\pi,\alpha,w) \in \cM_{p+q}(\cI)$ has $n$ vertices, we define $\MM_p^{\BU} = (\pi', \alpha',w')$ as the $\cI$-map in $\cM_{q}(\cI)$ with $n+1$ vertices where the first $p$ boundary edges of $\MM$, say $(e_1,\ldots, e_p)$ are wired to a new vertex $v = n+1$ of degree $p$, $w'_v=\SC$ and with $\partial v = (f_1,\ldots,f_p)$ such that $\alpha'(f_i) = e_i$. We often denote $\MM^{\BU}$ instead of $\MM_p^{\BU}$ for ease of notation.
Similarly, if $n \geq 2$ and $v \in V(\MM)$ is such that $\deg(v) = p$, we define $\MM^{\backslash v} \in \cM_{q} (\cI)$ as the $\cI$-map on $n-1$ vertices where $v$ has been removed and the directed edges $\alpha(\partial v)$ are matched to the first $p$ boundary edges (in a given fixed order). 
Finally, if $\MM \in \cM_q$ and $\sigma \in \mathrm{S}_q$ we denote by $\MM.\sigma$ the map where the boundary edges have been permuted by $\sigma$.

If $\MM = (\pi,\alpha,w) \in \cM_q(\cI)$ is an $\cI$-map, recall that $\MM (\cA^N) \in \cE^N_q$ is the corresponding action on $\cA^N$.  Finally, we set for any continuous function $f : \cE^N_q \to \dC$ such that the expression below is integrable:
$$
\dE_N [ f( \MM ) ] = \dE [ f ( \MM ( \cA^N ) )  ] .
$$

\begin{proposition}
\label{prop:SD1}
In the Gaussian case, for any connected $\cI$-map $\MM \in \cM_p(\cI)$, we have
$$
\dE_N [ \MM^\BU ] = \frac{1}{(p-1)!} \sum_{v,\sigma} \dE_N    [ (\MM.\sigma)^{\backslash v} ] + O ( \frac{1}{N} ), 
$$
where the sum is over all $v \in V(\MM)$ such that $w_v = \SC$, all permutations $\sigma \in \mathrm{S}_p$ such that  $(\MM.\sigma)^{\backslash v}$  has $p$ connected components  (this sum might be empty).
\end{proposition}

\begin{proposition}
\label{prop:var1}
In the Gaussian case, for any $\cI$-map $\MM \in \cM_0(\cI)$, we have
$$
\dE_N [ |\MM - \dE_N \MM  |^2] = O ( \frac{1}{N} ). 
$$
Moreover, for any $\cI$-map $\MM \in \cM_0(\cI)$, with connected components $(\MM_1,\ldots,\MM_{\gamma})$ we have
$$
\dE_N [ \MM ] = \prod_{i=1}^{\gamma} \dE_N [ \MM_i ]  + O ( \frac{1}{N} ). 
$$
\end{proposition}

Theorem \ref{th:free1} in the Gaussian case could be obtained as a corollary of Proposition \ref{prop:SD1} and Proposition \ref{prop:var1}. We will however give an alternative argument in Subsection \ref{subsec:free4}.

\subsection{Proof of Proposition \ref{prop:SD1} and Proposition \ref{prop:var1}}

These propositions are based on the following lemma which is a consequence of the Gaussian integration by part formula: if $Z$ is a real Gaussian variable, for all $f : \dR \to \dC$ in $C^1(\dR)$ such that $\dE |f'(Z)| < \infty$ then 
\begin{equation}\label{eq:gIPP}
    \dE Z f (Z) = \dE Z^2 \dE f'(Z).
\end{equation}

\begin{lemma}
\label{le:GItensor}
For any $\MM \in \cM_p(\cI)$ with $\gamma$ connected components, we have
$$
\dE_N [ \MM^\BU ] = \frac{1}{(p-1)!} \sum_{v,\sigma} \frac{N^{\gamma(v,\sigma)}}{N^{\gamma + p -1}} \dE_N    [ (\MM.\sigma)^{\backslash v} ],
$$
where the sum runs over all $v$ such that $w_v = \SC$, all permutations $\sigma \in \mathrm{S}_p$  and $\gamma(v,\sigma)$ is the number of connected components of  $(\MM.\sigma)^{\backslash v}$.
\end{lemma}
\begin{proof}
    For a given $i \in \INT{N}^p$, the number of permutations $\sigma \in \mathrm{S}_p$ such that $i_{\sigma} = i$ is $\cS_i =  p! / \cP_i$, where  $\cP_i$ is the number of elements in the equivalence class of $i$. Let $\MM \in \cM_p(\cI)$ and let $u$ be the newly added vertex in $\MM^\BU$. Using \eqref{eq:varX}, we compute using the Gaussian integration by part \eqref{eq:gIPP}:
    \begin{align*}
        \dE_N [ \MM^\BU ] &= \frac{1}{N^{\gamma}} \sum_{i \in \INT{N}^{E(\MM^\BU)}} \dE_N [W_{i_{\partial u}}\prod_{v \in V(\MM^\BU)\setminus \{u\}} (x_{v})_{i_{\partial v}}] \\
        &= \frac{1}{N^{\gamma}} \sum_{i \in \INT{N}^{E(\MM^\BU)}} \sum_{v : w_v=\SC \atop  \exists \sigma \in \mathrm{S}_p :  i_{\partial v} =   i_{\sigma(\partial u)}} \dE_N [W^2_{i_{\partial u}}] \dE_N[\prod_{t \in V(\MM^\BU)\setminus \{u,v\}} (x_{v})_{i_{\partial t}}] \\
        &= \frac{1}{N^{\gamma+p-1}} \sum_{v, \sigma } \sum_{i \in \INT{N}^{E(\MM^\BU)}} \frac{ \mathbf{1}_{i_{\partial v}=i_{\sigma(\partial u)}} }{ \cS_{i_{\partial u} }}\dE_N [X^2_{i_{\partial u}}] \dE_N[\prod_{t \in V(\MM^\BU)\setminus \{u,v\}} (x_{v})_{i_{\partial t}}] \\
        &=\frac{1}{(p-1)!} \sum_{v, \sigma } \frac{1}{N^{\gamma+p-1}} \sum_{i \in \INT{N}^{E(\MM^\BU)}} \mathbf{1}_{i_{\partial v}=i_{\sigma(\partial u)}}  \dE_N[\prod_{t \in V(\MM^\BU)\setminus \{u,v\}} (x_{v})_{i_{\partial t}}] \\
        &= \frac{1}{(p-1)!} \sum_{v,\sigma} \frac{N^{\gamma(v,\sigma)}}{N^{\gamma + p -1}} \dE_N    [ (\MM.\sigma)^{\backslash v} ]
    \end{align*}
with $\gamma(v,\sigma)$ the number of connected components of  $(\MM.\sigma)^{\backslash v}$. 
\end{proof}

\begin{lemma}\label{le:TENSUbounded}
    For any $\cI$-map $\MM \in \cM_0(\cI)$, we have 
  $$
  \dE_N [\MM] = O(1).
  $$  
\end{lemma}

\begin{proof}
By induction on the number $t$ of vertices $v$ of $\MM$ such that $w_v = \SC$. If $t=0$ it is true by Assumption (A1) and if $t=1$ then $\dE_N [\MM] = 0$. Otherwise, if $t\geq 2$ we apply Lemma \ref{le:GItensor} to delete two vertices such that $w_v = \SC$ and then conclude by induction. Another proof of this lemma could be achieved by Wick calculus.
\end{proof}

\begin{proof}[Proof of Proposition \ref{prop:SD1}]
The Proposition \ref{prop:SD1} is an immediate corollary of Lemma \ref{le:GItensor} and Lemma \ref{le:TENSUbounded}.
\end{proof}

\begin{proof}[Proof of Proposition \ref{prop:var1}] 
We first prove the second claim of the proposition: $\dE_N [ \MM ] = \prod_{i=1}^{\gamma} \dE_N [ \MM_i ]  + O ( 1/ N )$. The proof is by induction on the number $2t$ of vertices such that $w_v = \SC$ (if there is an odd number of such vertices, the expectation is zero).  If $t = 0$, all tensors are deterministic and the claim is trivial. Otherwise, $t \geq 1$ and assume that the statement holds for $t-1$. 

 Let $\MM \in \cM_0(\cI)$ with $\gamma$ connected components $( \MM_1, \ldots, \MM_\gamma)$ and $2t$ vertices such that $w_v = \SC$. We write $\MM=\widetilde \MM^\BU$ for some $\widetilde \MM \in \cM_p(\cI)$ with connected components $(\widetilde \MM_1, \ldots, \widetilde \MM_\gamma)$ such that $  \MM_1 = \widetilde \MM_1^\BU$ and $\widetilde \MM_j = \MM_j$ for $ 2 \leq j \leq \gamma$. An application of Lemma \ref{le:GItensor} gives:
$$
\dE_N [ \MM ] =  \frac{1}{(p-1)!} \sum_{v,\sigma} \frac{N^{\gamma(v,\sigma)}}{N^{\gamma + p -1}} \dE_N    [ (\widetilde \MM.\sigma)^{\backslash v} ],
$$
where  the sum runs on all $v \in V (\widetilde \MM)$ such that $w_v = \SC$, all  $\sigma \in \mathrm{S}_p$  and $\gamma(v,\sigma)$ is the number of connected components of  $(\widetilde \MM.\sigma)^{\backslash v}$. If $v \notin V(\widetilde \MM_1)$ then $\gamma(v,\sigma) \leq p + (\gamma-2)$ (because the two involved connected components will generate at most $p$ connected components after the rewiring). Hence, from Lemma \ref{le:TENSUbounded}, the corresponding contribution is $O(1/N)$. Since $ (\widetilde \MM.\sigma)^{\backslash v}$ has $2(t-1)$ vertices such that  $w_v = \SC$,  we may apply the induction hypothesis and deduce 
$$
\dE_N [ \MM ] = \left( \frac{1}{(p-1)!} \sum_{v,\sigma} \frac{N^{\gamma(v,\sigma)}}{N^{\gamma + p -1}}  \dE_N    [ (\widetilde \MM_1.\sigma)^{\backslash v} ] \right) \prod_{j = 2}^\gamma  \dE_N    [ \MM_j ] + O \left(\frac 1 N \right). 
$$
where $v \in V(\widetilde \MM_1)$ such that $w_v = \SC$. However, from Lemma \ref{le:TENSUbounded} again 
$$
\dE [\MM_1] = \dE [\widetilde \MM^\BU_1]= \frac{1}{(p-1)!} \sum_{v,\sigma} \frac{N^{\gamma(v,\sigma)}}{N^{\gamma + p -1}} \dE_N    [ (\widetilde \MM_1.\sigma)^{\backslash v} ].
$$
This concludes the proof of the inductive step of the second claim.

We next prove the first claim. It is an easy consequence of the second claim. We take $\MM \in \cM_0(\cI)$ and write
$$
\dE_N [ |\MM - \dE_N \MM  |^2] = \dE_N[\MM \sqcup  \MM_{\text{copy}}] -  |\dE_N [\MM] |^2, 
$$
where $\MM\sqcup   \MM_{\text{copy}}$ is the $\cA$-map corresponding to the disjoint union of two copies of $ \MM$ where in the second copy $\MM_{\text{copy}}$, the color map $w_{\text{copy}}$ is changed to its complex conjugate: if, in $\MM$, $w(v) = i_v$ then $w_{\text{copy}}(v) = i'_v$ where $i'_v \in \mathcal I_0$ is the index such that $A_{i'_v}^N  = \bar A^N_{i_v}$ (recall that we have assumed that $\cA^N_0$ is stable by conjugation). By construction  $\dE_N [\MM_{\text{copy}}]  = \overline{ \dE_N [\MM] }$. We apply the second claim to  $\dE_N[\MM \sqcup  \MM_{\text{copy}}] $ and get $ \dE_N[\MM \sqcup  \MM_{\text{copy}}] = |\dE_N [\MM] |^2 + O (1/N)$. The conclusion follows.
\end{proof}



\subsection{Schwinger-Dyson equations for Haar orthogonal matrices: proof of Theorem \ref{th:free2}}

In this paragraph,  we set $N \geq 1$ be an integer and $U = U_N$ is Haar distributed on the orthogonal group $\mathrm{O}(N)$.

We start with a simple observation. Consider a map $\MM \in \cM_0$, a vertex $u \in V(\MM)$ of degree $2$ and let $\MM' \in \cM_0$ be the map where the neighboring edges $\partial u = (e_1,e_2)$ of $u$ have been permuted to $(e_2,e_1)$. Let $x \in \cE_{\MM}$. We obviously have $\MM(x) = \MM' (x')$ where $x'_v = x_v$ for all $v \ne u$ and $x'_u = x_u^{\intercal}$ is the matrix transpose of $x_u \in \cE^N_2$. Consequently, it is sufficient to consider maps with vertices attached to $U_N \in \mathrm{O}(N)$ only (since $U_N^* = U_N^{\intercal}$). 

We let $\cI  =\cI_0 \cup \{ \HU\}$ and $\ell : \cI \to \{1,2,\ldots \}$ such that $\ell = \ell_0$ on $\cI_0$ and $\ell(\HU) = 2$. We assume without loss of generality that there is an element $1 \in \cI$ such that $A_1^N = I_N \in \cE^N_2$ is the identity matrix. We define $A^N_\HU  = U_N$ and $\cA^N = \cA_0^N \cup \{ U_N  \} = \{ A_i^N : i \in \cI \}$.

As above, if $\MM = (\pi,\alpha,w) \in \cM_q(\cI)$ is an $\cI$-map, recall that $\MM (\cA^N) \in \cE^N_q$ is the corresponding action on $\cA^N$.  Finally, we set for any continuous function $f : \cE^N_q \to \dC$ such that the expression below is integrable:
$$
\dE_N [ f( \MM ) ] = \dE [ f ( \MM ( \cA^N ) )  ] .
$$

The following proposition is a version of Proposition \ref{prop:SD1} for Haar orthogonal matrices where chromatic switches appear. If $\MM  = (\pi,\alpha,w) \in \cM(\cI)$ and $e \in \partial v$, we set $w_e = w_v$.

\begin{proposition}\label{prop:SD2}
Let $\MM =(\pi,\alpha,w) \in \cM_0(\cI)$ with $\gamma$ connected components, $\veps \in \{1,2\}$,  $u \in V(\MM)$ be such that $w_u = \HU$ and set $\partial u = (f_1,f_2)$. We have
$$
\dE_N [\MM] = \sum_{v} \dE_N [\MM_{u,v,\veps}^+] -    \sum_{v} \dE_N [\MM_{u,v,\veps}^-] + O(\frac 1 N),
$$
where the sums are over all $v \ne u \in V(\MM)$ in the connected component of $u$, $\partial v = (g_1,g_2)$
such that $w_v = \HU$  and $\MM_{u,v,\veps}^\pm < \MM$ has $\gamma+1$ connected components. In the first sum, $\MM_{u,v,\veps}^+$ is obtained from $\MM$ by the chromatic switch which connects $f_\veps$ and $g_\veps$ (and $\alpha(f_\veps)$ with $\alpha(g_\veps)$). In the second sum  $\MM_{u,v,\veps}^-$ is obtained from $\MM$ by the chromatic switch which connects $f_\veps$ and $\alpha(g_\veps)$ (and $\alpha(f_\veps)$ with $g_\veps)$). Finally, $\MM_{u,v,\veps}^+(\cA^N)$ is left unchanged if we replace $w_v = w_u = \HU$ by the identity.
\end{proposition}

\begin{proposition}
\label{prop:var2}
For any $\cI$-map $\MM \in \cM_0(\cI)$, we have
$$
\dE_N [ |\MM - \dE_N \MM  |^2] = O ( \frac{1}{N} ). 
$$
Moreover, for any $\cI$-map $\MM \in \cM_0(\cI)$, with connected components $(\MM_1,\ldots,\MM_{\gamma})$ we have
$$
\dE_N [ \MM ] = \prod_{i=1}^{\gamma} \dE_N [ \MM_i ]  + O ( \frac{1}{N} ). 
$$
\end{proposition}

Theorem \ref{th:free2} is a corollary of Proposition \ref{prop:SD2} and Proposition \ref{prop:var2}.

\begin{proof}[Proof of Theorem \ref{th:free2}] We should check \eqref{eq:asfree}. Observe that for any finite collection $\cB_0^N = (\mathfrak{p}_j (A^N)) _{j \in \mathcal J_0}$, with $\mathfrak{p}_j \in \mathbb C[\cM_{q_j} (\mathcal I_0)]$, then Assumption (A1) also holds for  $\cB_0^N$ by property (S) and (L). Hence up to replacing the finite collection $\cA^N_0  = (A^N_{i})_{i \in \mathcal I_0}$ by any other finite collection $(\mathfrak{p}_j (A^N)) _{j \in \mathcal J_0}$, it is enough to check that, in probability, $\lim_{N \to \infty} \MM ( \cA^N) = 0$ for any map $\MM \in \cM_0 (\cI)$ satisfying conditions (P1) and (P2) in the definition of freeness. 

To this end, from Proposition \ref{prop:var2} and Markov inequality, it suffices to check the asymptotic freeness in expectation. We may then iterate Proposition \ref{prop:SD2} and \ref{prop:var2}. Fix an $\cI$-map $\MM$. By Proposition \ref{prop:SD2} and Proposition \ref{prop:var2} applied recursively, we have $\dE_N [\MM] = \sum t(\widehat \MM) \dE_N [\widehat \MM] + O(1/N)$ where the sum is over all $\widehat \MM$  satisfying conditions (P1)-(P2) in the definition of freeness and $t(\widehat \MM) \in \dZ$. On the other hand, if such $\widehat \MM$ has a minimal non-monochromatic or a monochromatic centered connected component, we find by Proposition \ref{prop:var2} and Proposition \ref{prop:SD2} (for the minimal non-monochromatic case), that $\dE_N [\widehat \MM] = O(1/N)$.
The conclusion follows.
\end{proof}

\subsection{Proof of Proposition \ref{prop:SD2} and Proposition \ref{prop:var2}}

We start by checking that $\dE_N [\MM]$ is uniformly bounded for all $\cI$-maps.
\begin{lemma}
\label{le:Ubounded}
  For any $\cI$-map $\MM \in \cM_0(\cI)$, we have 
  $$
  \dE_N [\MM] = O(1).
  $$
\end{lemma}

Lemma \ref{le:Ubounded} can be established as a consequence of Weingarten calculus for the Orthogonal group. For an even integer $k \geq 1$, we denote by $\cP_k$ the set of pairings of $\INT{k}$ (that is the permutations $\sigma \in \mathrm{S}_k$ such that $\sigma^2 $ is the identity and without fixed point). 

\begin{proposition}[Corollary 3.4 in \cite{MR2217291}]\label{prop:Wg1}
Let $N \geq 1$ be an integer, $U_N$ be Haar distributed on the orthogonal group $\mathrm{O}(N)$. Let $k$ be  even. There exists a function $\sigma \to \Wg_N (\sigma)$ on  $\mathrm{S}_k$ such that, for any $i, j$ in $\INT{N}^k$.   
\begin{equation*}
\EE \PAR{\prod_{t=1}^k (U_N)_{i_t,j_t}}=\sum_{p,q\in \cP_k}\delta_{p}(i)\delta_{q}(j)\Wg_N (p q^{-1}),
\end{equation*}
and $\delta_p(i) \in \{0,1\}$ is zero unless $i_{l} = i_{p(l)}$ for all $l \in \INT{k}$.
\end{proposition}

The Weingarten function $\Wg_N(\sigma)$ admits various representations. Here, we shall only use the following basic estimate. For $\sigma\in \mathrm{S}_k$, we denote by $|\sigma| = k - \ell(\sigma)$ where $\ell(\sigma)$ is the number of disjoint cycles in the cycle  decomposition of $\sigma$ ($|\sigma|$ is also the minimal number $m$ such that $\sigma$ can be written as a product  of $m$ transpositions). In particular $(-1)^{|\sigma|}$ is the signature of $\sigma$. 

\begin{lemma}[Theorem 3.13 in \cite{MR2217291}]\label{le:Wg2}
For any fixed $k$ and $\sigma \in \mathrm{S}_k$, we have 
$$\Wg_N(\sigma) = (-1)^{|\sigma|} N^{-k/2-|\sigma|/2} \left( 1+ O (\frac{1}{N})\right).  $$
\end{lemma}

\begin{proof}[Proof of Lemma \ref{le:Ubounded}]
Let $\gamma$ be the number of connected components of $\MM$ and let $k$ be the number of vertices $v \in V(\MM)$ such that $w_v = \HU$. We denote by $V_{\HU}(\MM) = \{ v_1, \ldots, v_k\}$ these vertices and set $\partial v_t = (e_t,f_t)$ for $t \in \INT{k}$.  We set $V_0 (\MM) = V (\MM) \backslash V_{\HU} (\MM)$. By construction $\MM(\cA^N)$ is a weighted sum of product of $k$ entries of $U_N$: $\prod_{t=1}^k (U_N)_{i_{\bar e_t}, i_{\bar f_t}}$, where $\bar e \in E(\MM)$ is the edge $e \in \vec E(\MM)$ belongs to . Also, since  $U_N U_N^{\intercal} = I_N$, at the cost of lowering $k$, we can assume without loss of generality that $\MM$ has no edge connecting $e_t$ and $e_{s}$ or $f_t$ and $f_s$. 

By Proposition \ref{prop:Wg1}, we may write 
$$
\dE_N [ \MM]  =  \frac{1}{N^\gamma}  \sum_{p,q\in \cP_k}\Wg_N (p q^{-1})  \sum_{i  \in \INT{N}^{E(\MM)}}  \delta_{p,q}(i) \prod_{v \in V_0(\MM)} (A^N_{w(v)})_{i_{\partial v}},
$$
where $\delta_{p,q}(i) = 1$ if for all $t \in \INT{k}$, $i_{\bar e_t} = i_{\bar e_{p(t)}} $ and $i_{\bar f_t} = i_{\bar f_{q(t)}} $ and $\delta_{p,q}(i) = 0$ otherwise. In view of Lemma \ref{le:Wg2}, to prove Lemma \ref{le:Ubounded}, it is sufficient to prove that, for any  $(p,q) \in \cP_{k}$,
$$
N^{-\gamma - k/2 - |pq^{-1}|/2} \sum_{i  \in \INT{N}^{E(\MM)}}  \delta_{p,q}(i) \prod_{v \in V_0(\MM)} (A^N_{w(v)})_{i_{\partial v}}  = O(1)
$$

Given $(p,q) \in \cP_{k}$, we  define below an $\cA^N_0$-map, say $\MM_{p,q}$, with vertex set $V_0(\MM)$ such that 
\begin{equation}\label{eq:MMpq}
\sum_{i  \in \INT{N}^{E(\MM)}}  \delta_{p,q}(i) \prod_{v \in V_0(\MM)} (A^N_{w(v)})_{i_{\partial v}} = N^{\alpha(p,q)} \MM_{p,q}(\cA^N_0),
\end{equation}
where $\alpha(p,q) \in \mathbb Z$. By assumption (A1), $|\MM_{p,q}(\cA_0^N)| = O(1)$. Hence the proof will be complete if we manage to check that $\alpha(p,q) \leq k/2 + \gamma + |pq^{-1}|/2$.

This map $\MM_{p,q}$ is defined as follows. Let us say that $e \in \vec E(\MM)$ is of type $0$ if $e \in \partial v$ with $w_v \in \mathcal I_0$ and, otherwise, for $t \in \INT{k}$, we say that $e = e_t$ is of type $\HU^+$ and $e =f_t$ is of type $\HU^-$. In $\MM_{p,q}$ the directed edge set is the subset $\vec E(\MM)$ of type $0$ elements. Next, we say  that $e,f \in E(\MM)$ are neighbors if either they form an edge (that is $\alpha(e) = f$) or, if $e$ and $f$ are of type $\HU^{\pm}$, if either $p(e) = f$ (type $\HU^+$) or if $q(e) = f$ (type $\HU^-$). This defines a graph $G_{p,q}$ with vertex set $\vec E(\MM)$ where vertices of type $0$ have degree $1$ and vertices of type $\HU^{\pm}$ have degree $2$ (recall that we have assumed that there is no  edge in $\MM$ between two type $\HU^+$ neither between two type $\HU^-$). Hence each connected component is either cycle with no type $0$ edge or a segment with exactly $2$ type $0$ edges at its endpoints. In this last case, we put an edge in $\MM_{p,q}$ between these two edges. This defines $\MM_{p,q}$ in \eqref{eq:MMpq}.

From what precedes, the integer $\alpha(p,q)$  in \eqref{eq:MMpq} is equal to $\beta(p,q)$, the number of cycles of $G_{p,q}$, plus $\gamma(p,q)$, the number of connected components of $\MM_{p,q}$.   
We prove below the bound
\begin{equation}\label{eq:boundcpq}
\alpha(p,q)  = \beta(p,q) + \gamma(p,q) \leq \frac{k}{2} + \gamma + \frac{|pq^{-1}|}{2}.
\end{equation}
To this end, let $\Gamma_{p,q}$ be the (multi)graph on $\INT{k}$ whose edges are the pairs of $p$ and the pairs of $q$. Every vertex of $\Gamma_{p,q}$ has degree $2$, so $\Gamma_{p,q}$ is a disjoint union of cycles of even length, say $c(\Gamma_{p,q})$ of them. First, $\ell(pq^{-1}) = 2 c(\Gamma_{p,q})$: on a cycle of $\Gamma_{p,q}$ of length $2\ell$, the permutation $pq^{-1} = pq$ preserves each of the two alternating classes of vertices of this cycle and acts on each of them as a cyclic permutation of order $\ell$. Hence
$$
|pq^{-1}| = k - 2\,c(\Gamma_{p,q}).
$$

\emph{Step 1: the case $q = p$.} Let $C$ be the graph obtained from $\MM$ by merging, for each pair $\{s,t\}$ of $p$, the two vertices $v_s$ and $v_t$ into a single vertex $u_{\{s,t\}}$ of degree $4$, incident to $e_s, f_s, e_t, f_t$. Merging vertices does not increase the number of connected components, so $C$ has at most $\gamma$ connected components. Let $C'$ be obtained from $C$ by splitting each $u_{\{s,t\}}$ into two vertices of degree $2$, one incident to $e_s, e_t$ and the other incident to $f_s, f_t$. Each of these $k/2$ splittings increases the number of connected components by at most one, so that $C'$ has at most $\gamma + k/2$ connected components. On the other hand, the number of connected components of $C'$ is equal to $\beta(p,p) + \gamma(p,p) = \alpha(p,p)$. Indeed, $C'$ is obtained from $G_{p,p}$ by contracting each pairing edge $\{e_s, e_{p(s)}\}$ (resp.\ $\{f_t, f_{p(t)}\}$) to the corresponding split vertex and by identifying the type $0$ directed edges attached to a common vertex of $V_0(\MM)$. The cycle components of $G_{p,p}$ contain no type $0$ directed edge and are mapped bijectively to the components of $C'$ which avoid $V_0(\MM)$, while the segment components of $G_{p,p}$ become the edges of $\MM_{p,p}$, so that the components of $C'$ meeting $V_0(\MM)$ are in bijection with the components of $\MM_{p,p}$. This proves \eqref{eq:boundcpq} for $q = p$.

\emph{Step 2: induction on $|pq^{-1}|$.} Assume $p \ne q$. We pick $s \in \INT{k}$ with $q(s) \ne p(s)$ and let $q' \in \cP_k$ coincide with $q$ except on the two pairs $\{s, q(s)\}$ and $\{p(s), q(p(s))\}$, which are replaced by $\{s, p(s)\}$ and $\{q(s), q(p(s))\}$ (these four elements are pairwise distinct). In $\Gamma_{p,q'}$, the cycle of $\Gamma_{p,q}$ containing $s$ splits into the double edge $\{s, p(s)\}$ and one shorter cycle, the other cycles being unchanged. Hence $c(\Gamma_{p,q'}) = c(\Gamma_{p,q}) + 1$, that is, $|p(q')^{-1}| = |pq^{-1}| - 2$. By construction, $G_{p,q'}$ is obtained from $G_{p,q}$ by a switch of the two pairing edges $\{f_s, f_{q(s)}\}$ and $\{f_{p(s)}, f_{q(p(s))}\}$. We claim that
$$
\alpha(p,q) \le \alpha(p,q') + 1.
$$
Iterating, after $|pq^{-1}|/2$ such steps we reach the pairing $p$, so that
$\alpha(p,q) \le \alpha(p,p) + |pq^{-1}|/2$, which together with Step 1 proves \eqref{eq:boundcpq}.

To prove the claim, we check more generally that a switch of two pairing edges of $G_{p,q}$ changes the quantity $\alpha = \beta + \gamma$ by at most one. Recall that $G_{p,q}$ is a disjoint union of cycles and of segments whose endpoints are the type $0$ directed edges. We distinguish according to the components containing the two switched edges. If they lie in two distinct components, at least one of which is a cycle, the switch merges them into a single component of the same nature as the other one (two cycles give a cycle, a cycle and a segment give a segment with the same endpoints): $\beta$ decreases by one and $\MM_{p,q}$ is unchanged. If they lie in two distinct segments, say $x_0 \cdots x\, y \cdots x_1$ and $z_0 \cdots z\, w \cdots z_1$ with switched edges $\{x,y\}$ and $\{z,w\}$, the switch produces again exactly two segments (for instance $x_0 \cdots x\, z \cdots z_0$ and $x_1 \cdots y\, w \cdots z_1$): $\beta$ is unchanged, while the pairs of endpoints $\{x_0, x_1\}$ and $\{z_0, z_1\}$ are replaced by $\{x_0, z_\veps \}$ and $\{x_1, z_{1-\veps}\}$, $\veps \in \{0,1\}$. In other words, $\MM_{p,q}$ is changed by a switch and its number of connected components changes by at most one. Finally, if the two switched edges lie in the same component, the switch either leaves it connected, in which case it remains a cycle or a segment with the same endpoint, or disconnects it into two cycles, or into a segment with the same endpoints and a cycle: $\beta$ increases by at most one and $\MM_{p,q}$ is unchanged. In every case $\alpha = \beta + \gamma $ changes by at most one, which proves the claim and completes the proof of \eqref{eq:boundcpq}. \end{proof}

We need to introduce yet another map operation. Fix $\veps \in \{1,2\}$ and set $\hat 1 = 2$,  $\hat 2 = 1$. Let $\MM  = (\pi,\alpha,w) \in \cM_{2+q}(\cI)$  and $v \in V(\MM)$ such that $\deg(v) = 2$. Letting $\partial v = (f_1,f_2)$ and $(e_1,e_2)$ being the first two boundary edges of $\MM$, we define $\MM_{v,\veps} = (\pi', \alpha',w')$ as the $\cI$-map in $\cM_{q}(\cI)$ obtained from $\MM$, by wiring $e_{\hat \veps}$  with $f_\veps$  and, if $f_\veps \ne e_\veps$,  $\alpha(f_\veps)$ with $e_\veps$ (that is $\alpha' (e_{\hat \veps}) = f_\veps$ and $\alpha'(\alpha(f_\veps)) = e_\veps$). If $f_\veps = e_\veps$, we create a new connected component composed with a single vertex with a loop edge with color the identity $1 \in \cI$. Finally, recall that if $\sigma \in S_{q}$ and $\MM  \in \cM_{q}$, we denote by $\MM.\sigma \in \cM_q$ the map where the boundary edges are permuted by $\sigma$. 

\begin{lemma}[Schwinger-Dyson equation for Haar orthogonal matrices]\label{le:Haarinv}
For any $\cI$-map $\MM \in \cM_2(\cI)$ and $\veps \in \{1,2\}$, we have
$$ \sum_{v : w_v = \HU,\sigma \in S_2} (-1)^{|\sigma|}   N^{\gamma(v,\sigma)}  \dE_N [(\MM.\sigma)_{v,\veps}]  = 0,$$
where $\gamma(v,\sigma)$  is the number of connected components of $(\MM.\sigma)_{v,\veps}$.
\end{lemma}
\begin{proof}
Assume $\veps = 2$. Let $B$ be a real anti-symmetric matrix of size $N$. For real $t$, we set $U(t) = U e^{tB} \in \mathrm{O}(N)$ and $\cA^N(t) = \cA_0^N \sqcup \{ U(t) \}$. By invariance of the Haar measure: we have the matrix identity
$$
\dE_N [ \MM ] = \dE [ \MM ( \cA^N) ]= \dE [ \MM ( \cA^N(t)) ].
$$
In particular, the derivative at $0$ of the map $t \mapsto \dE \MM ( \cA^N(t))$ is $0$. We find, for any real anti-symmetric $B$:
\begin{equation}\label{eq:HInv}
    0 = \sum_{v : w_v = \HU} \dE_N [\MM_{v,B}], 
\end{equation}
where for $v \in V(\MM)$ such that $w_v = \HU$,  $\MM_{v,B}$ is obtained from $\MM$ by inserting a new vertex of degree $2$, say $b$, between $v$ and $u_2$, the neighbor of $v$ connected to its second edge of $\partial v$. The first edge of $b$ is connected to $v$. This vertex $b$ is colored with the matrix $B$.
Say that the boundary edges of $\MM$ are $(1,2)$, so that  $\MM_{i_1,i_2}$ are the corresponding matrix entries. In \eqref{eq:HInv}, we consider the matrix $B = e_{i_1} \otimes e_{i_2} - e_{i_2} \otimes e_{i_1}$. We sum over $i_1,i_2$ and get the requested expression.

For $\veps = 1$, the proof is identical but consider instead $ U(t) =  e^{tB}U$ (or use that $U$ and $U^\intercal$ have the same distribution).
\end{proof}

We remark that in the Haar unitary case, a stronger form of Lemma \ref{le:Haarinv} holds, where the sum is only over $v$, not over $\sigma \in S_2$ (but we have to consider the matrix $U_N$ and its complex conjugate $\bar U_N$). Indeed, it suffices to consider in \eqref{eq:HInv} the anti-Hermitian matrices $B = \mathrm{i} e_ {i_1} \otimes e_{i_2} + \mathrm{i} e_{i_2} \otimes e_{i_1}$ and $B = e_{i_1} \otimes e_{i_2} - e_{i_2} \otimes e_{i_1}$ and combine their contributions to cancel $ e_{i_2} \otimes e_{i_1}$.

\begin{proof}[Proof of Proposition \ref{prop:SD2}] 
Assume for example $\veps = 2$. Let $\partial u = (f_1,f_2)$, $\alpha(f_2) = e_1$. Let $\tilde \MM \in \cM_2(\cI)$ with boundary edges $(e_1,f_2)$ be obtained from $\MM$ by removing the edge between $e_1$ and $f_2$. In particular $\MM^{+\HU}  = \tilde \MM^{\CL}$.  We apply Lemma \ref{le:Haarinv} to $\tilde \MM$ such that $\MM^{+\HU}  = \tilde \MM^{\CL}$. By construction, $\tilde \MM_{u,\veps}$ has $\gamma+1$ connected components and $\tilde \MM_{u} (\cA_N) = \MM(\cA_N)$ (the connected component with a single vertex with the identity does not contribute). For an element $v \in V(\MM)$ and $\sigma \in \mathrm{S_2}$,  $(\tilde \MM.\sigma)_{v,\veps}$ will have $\gamma+1$ connected components only if $v$ is in the connected component of $u$ and if the removal of $v$ disconnects this component of $\tilde \MM$. The proposition is then a consequence of Lemma \ref{le:Haarinv} and using, for $\sigma = (1 2) $, the identities  $U_N U_N^* = U_N^* U_N = I_N$ and Lemma \ref{le:Ubounded}.
\end{proof}

\begin{proof}[Proof of Proposition \ref{prop:var2}]
The proof is similar as the one of Proposition \ref{prop:var1} using the Schwinger-Dyson equations for Haar orthogonal matrices instead of the ones for symmetric random tensors. As argued there, it is enough to prove the second claim: $
\dE_N [ \MM ] = \prod_{i=1}^{\gamma} \dE_N [ \MM_i ]  + O ( 1/N )$, where $\MM_1, \ldots, \MM_\gamma$ are the connected components of $\MM$. We do a recursion on the maximal number, say  $t$, of vertices $v \in V(\MM)$ such that $w_v = \HU$ in any connected component of $\MM$. If $t = 0$, there is nothing to prove.

For $ t\geq 1$, assume that we have proved that for all maps with at most $t-1$ $\HU$-vertices in each of its connected components,  we  have $
\dE_N [ \MM ] = \prod_{i=1}^{\gamma} \dE_N [ \MM_i ]  + O ( 1/N )$ (where the constants hidden in the $O(\cdot)$ depend on $\MM$ as usual). Take a map $\MM$ with connected components $(\MM_1,\ldots, \MM_\gamma)$ and assume that $\MM_i$ has at most $t$ $\HU$-vertices for each $i \in \INT{\gamma}$. Let $u \in V(\MM_1)$ with $w_{u} = \HU$ and $\veps \in \{1,2\}$. We apply Proposition \ref{prop:SD2} to $u$. We notice that in each new connected component of $\MM^{\pm}_{u,v,\veps}$ possibly appearing in Proposition \ref{prop:SD2} there are at most $t-1$ $\HU$-vertices, the other components are $\MM_2, \ldots, \MM_{\gamma}$. In each $\MM^{\pm}_{u,v,\veps}$, we repeat the same operation by picking a $\HU$-vertex in $\MM_2$. We continue iteratively with $\MM_3, \ldots, \MM_\gamma$. We find that $\dE_N[\MM]$ is equal to $O(1/N)$ plus  a finite sum of terms $\pm \dE_N [\widetilde \MM ]$ where each $\widetilde \MM$ has at most $t-1$ $\HU$-vertices in each connected component. It remains to use the induction hypothesis. \end{proof}

\subsection{Schwinger-Dyson equations for unitarily invariant families: proofs of Theorem \ref{th:free1} (Gaussian case) and Theorem \ref{th:free3}} \label{subsec:free4}

 For $\veps \in \{0,1\}$,  let $\cI_\veps$ be a finite set and $\ell_\veps : \cI_\veps \to \{1, 2 , \ldots \}$.  We set $\cI = \cI_0 \sqcup \cI_1$ and define $\ell : \cI \to \{1,2,\dots\}$ whose restriction to $\cI_\veps$ is $\ell_\veps$. Let $\cA^N_0 = (A^N_{0,i} )_{i \in \cI_0}$ and $\cA^N_1 = (A^N_{1,i} )_{i \in \cI_1}$ be two finite deterministic families such that for any $i \in \cI$ and $N$, $A^N_{i} \in \cE^N _{\ell(i)}$. We assume that $\cA^N_0 $ and $\cA^N_1$ both satisfy (A1).  Let $U_N \in \mathrm{O}(N)$ be Haar distributed. For ease of notation, we set $\tilde \cA^N_1 =  \cA^N_1 \cdot U_N^{\#}$.  We set $\cA^N = \cA^N_0 \sqcup \tilde \cA^N_1$ and set, as above, for $\MM \in \cM_q(\cI)$ and $f : \cE^N_q \to \dC$ continuous we set
 $
 \dE_N [ f (\MM) ] = \dE [f ( \MM ( \cA^N))]$ where the expectation is with respect to $U_N$ (provided that it is integrable). The next proposition is the Schwinger-Dyson equation for unitarily invariant families. Recall that if $\MM  = (\pi,\alpha,w) \in \cM(\cI)$ and $e \in \vec E(\MM) \cap \partial v$, we set $w_e = w_v$.

\begin{proposition}\label{prop:SD4}
Let $\MM =(\pi,\alpha,w) \in \cM_0(\cI)$ connected and  $e  \in \vec E(\MM)$ be such that $w_e \in \cI_1$,  $w_{\alpha(e)} \notin \cI_1$. We have
$$
\dE_N [\MM] = \sum_{f} \dE_N [\MM_{e,f}^+] -    \sum_{f} \dE_N [\MM_{e,f}^-] + O(\frac 1 N),
$$
where the sums are over all $f \ne e \in \vec E(\MM)$
such that $w_f \in  \cI_1$, $w_{\alpha(f)} \notin \cI_1$  and $\MM_{e,f}^\pm < \MM$ has two connected components. In the first sum, $\MM_{e,f}^+$ is obtained from $\MM$ by the chromatic switch which matches $e$ and $f$ (and $\alpha(e)$ with $\alpha(f)$). In the second sum  $\MM_{e,f}^-$ is obtained from $\MM$ by the chromatic switch which connects $e$ and $\alpha(f)$ (and $\alpha(e)$ with $f$).
\end{proposition}

\begin{proof}
Recall that $\tilde \cA^N_1 = \cA^N_1\cdot U_N^{\#}$. Let $\bar \cA^N_0 = \cA^N_0 \sqcup \cA^N_1$ and $\bar \cA^N = \cA^N_0 \sqcup \{ U_N \}$. We have $\MM(\tilde A^N) = \tilde \MM (\bar \cA_N)$ where $\tilde \MM$ is obtained from $\MM$ by adding a vertex, say $v_f$, of degree $2$ with a mark $\HU$ in the middle of all edges $(f,\alpha(f))$ with $w_f \in \cI_1$ and $w_{\alpha(f)} \notin \cI_1$. The second edge of $v_f$ is matched with $f$ and the first edge of $v_f$ is matched to $\alpha(f)$. We then apply Proposition \ref{prop:SD2} to $\tilde \MM$, $\veps = 2$ and $u = v_e$. 
\end{proof}

\begin{proof}[Proof of Theorem \ref{th:free3}]
The proof of Theorem \ref{th:free3} repeats  the proof of Theorem \ref{th:free2} with Proposition \ref{prop:SD2} replaced by Proposition \ref{prop:SD4}.
\end{proof}

\begin{proof}[Proof of Theorem \ref{th:free1}, Gaussian case]
Theorem \ref{th:free1} in the Gaussian case is a  consequence of Proposition  \ref{prop:SD4} and the fact that if $W^N$ is a Gaussian tensor of order $p \geq 1$, then $W^N$ and $W^N \cdot U_N^p $ have the same distribution, where $U_N \in \mathrm{O}(N)$ is an independent Haar uniform orthogonal matrix. In particular Proposition \ref{prop:SD4} applies to $W^N$ with $\cA^N_1 = \{ W^N \}$. We may then repeat the argument in the proof of Theorem \ref{th:free2} with Proposition \ref{prop:SD2} replaced by Proposition \ref{prop:SD4}.
\end{proof}


\subsection{Proof of Theorem \ref{th:free1b}}

We first note that, as observed in the proof of Theorem \ref{th:free2}, Assumption (A2) is stable by replacing $\cA^N_0$ by any finite collection $(\mathfrak{p}_j (\cA^N_0))_{j \in \mathcal J_0}$ with $\mathfrak{p}_j \in \dC [\cM_{q_j} (\cI_0)]$. Arguing as in the proof of Theorem \ref{th:free2}, it is then enough to check that, in probability, $\lim_{N \to \infty} \MM ( \cA^N) = 0$ for any map $\MM \in \cM_0 (\cI)$ satisfying conditions (P1) and (P2) in the definition of freeness. 

The proof is by comparison. In view of Theorem \ref{th:free1}, it is sufficient to prove that for any $\MM \in \cM_0(\cI)$, we have 
\begin{equation}\label{eq:tbp1}
\left| \dE_N [\MM] - \dE^{\gauss}_N [\MM] \right| = O(\frac 1   N),
\end{equation}
where $\dE^{\gauss}_N $ denotes the expectation with respect to the Gaussian random tensor with the same variance profile.  Indeed, \eqref{eq:tbp1} also implies that for any $\MM \in \cM_0(\cI)$, 
\begin{equation}\label{eq:tbp1var}
\dE_N [ \left|  \MM - \dE_N [\MM] \right|^2 ] =  O(\frac 1   N).
\end{equation}
To see this, the factorization property of Proposition \ref{prop:var1} for $\dE^{\gauss}_N$ and \eqref{eq:tbp1} imply that for any  $\MM \in \cM_0(\cI)$, with connected components $(\MM_1,\ldots, \MM_\gamma)$ we have: 
$$
\dE_N [\MM] = \dE^\gauss_N [\MM]  +  O(\frac 1 { N})= \prod_{i=1}^\gamma \dE^\gauss_N [\MM_i] +  O(\frac 1 {N}) =  \prod_{i=1}^\gamma \dE_N [\MM_i] + O(\frac 1 {N}),
$$
where we have applied Lemma \ref{le:TENSUbounded} and again \eqref{eq:tbp1} at the last step.
Next, the proof of Proposition \ref{prop:var1}  explains why the factorization property implies the vanishing variance property \eqref{eq:tbp1var}.

The remainder of the subsection is devoted to the proof of \eqref{eq:tbp1}. To this end, let $V_{\SC}$ be the subset of vertices of $\MM$ such that $w_v = \SC$. If $\PP$ is a partition of $V_{\SC}$, $\sigma = (\sigma_v)_{v \in V_\SC}$ is a sequence of permutations in $\mathrm{S}_p$ and $i \in \INT{N}^{E(\MM)}$, we define 
$
\delta_{\PP,\sigma} (i)  \in \{0,1\}$ as the indicator function that for all blocks $b = \{ u_1,\ldots,u_l\} \in \PP$, we have, for all $k \in \INT{l}$,
$
i_{\sigma_{u_k}(\partial u_k)} = i_{\sigma_{u_1} (\partial u_1)}.
$
Similarly, we define $\delta_{\PP}(i) = \max_{\sigma} \delta_{\PP,\sigma}(i)$ as the indicator that there exists a sequence of permutations $\sigma = (\sigma_v)_{v \in V_\SC}$ such that $ \delta_{\PP,\sigma}(i) = 1$. Also, we set 
$$
\mu_{\PP}(i) = \prod_{b \in \PP} \dE [ \prod_{k=1}^l X_{i_{\partial u_k}}  ]
$$
Notably, if $\delta_{\PP} (i) = 1$ then $ \mu_{\PP}(i) = \prod_{b \in \PP} \dE [ X^l_{i_{\partial u_1}}  ]$. In the same vein, we define $\delta'_{\PP}(i) \in \{0,1\}$ as the indicator  that $\{ i_{\partial v}\} \ne \{ i_{\partial u}\}$ for all $u ,v$ in different blocks of $\PP$.

Let $\gamma$ be the number of connected components of $\MM$ and $n = |V_{\SC}|$. Up to adding identity elements $I_N \in \cE^N_2$ to $\cA^N_0$ on some edges, we may assume without loss of generality that each connected component of $\MM$ contains at least one element in $\cA^N_0$ and that for all $v \in V_{\SC}$ and $e \in \partial v$, $\alpha (e) \in \partial u$ for some $u \notin V_{\SC}$. Recall that $W^N = X / N^{(p-1)/2}$. In the expression for $\MM$, we decompose the product  $\prod_{v \in V_{\SC}} X_{i_{\partial v}}$ over the distinct set of indices $\{ i_{\partial v}\}_{v \in V_{\SC}}$. We get
\begin{equation}\label{eq:decomp1}
\dE_N [\MM] = \frac{1}{N^{\gamma+\frac{(p-1)n}{2}}}\sum_{\PP} \sum_{ i \in \INT{N}^{ E(\MM)}}  \delta_{\PP}(i) \delta'_{\PP}(i) \mu_{\PP}(i) \prod_{v \in V(\MM) \backslash V_{\SC}} (x_v)_{i_{\partial v}},
\end{equation}
where $x_{v} \in \cA^0_N$ is the corresponding variable and the sum is over all partitions $\PP$ of $V_{\SC}$ with blocks of size at least $2$ (otherwise $\mu_{\PP}(i) = 0$ by the assumption $\dE X = 0$). If $\delta_{\PP} (i) = 1$, let $n_{\PP}(i)$ be the number of sequences of permutations $\sigma = (\sigma_v)_{v \in V_\SC}$ such that $ \delta_{\PP,\sigma}(i) = 1$. We may rewrite \eqref{eq:decomp1} as 
\begin{equation}\label{eq:decomp2}
\dE_N [\MM] = \frac{1}{N^{\gamma+\frac{(p-1)n}{2}}}\sum_{\PP,\sigma} \sum_{ i \in \INT{N}^{ E(\MM)}}  \delta_{\PP,\sigma}(i)\delta'_{\PP}(i) \frac{\mu_{\PP}(i) }{n_{\PP}(i)} \prod_{v \in V(\MM) \backslash V_{\SC}} (x_v)_{i_{\partial v}}.
\end{equation}

Next, if $\hat \MM \in \widehat \cM_0(\cI_0)$ and $i \in \INT{N}^{E(\hat \MM)}$ we set 
$$
\hat \MM(\cA^N_0,i) = \prod_{ v \in V(\hat \MM)} (x_v)_{i_{\partial v}},
$$
where $x_v \in \cA^N_0$ is the corresponding variable. Notably, if $\hat \gamma$ is the number of connected components of $\hat \MM$,
$$
\hat \MM(\cA_0^N) = \frac{1}{N^{\hat \gamma}} \sum_{i \in \INT{N}^{E(\hat \MM)}}\hat \MM(\cA^N_0,i).
$$
With this new notation, we can rewrite  \eqref{eq:decomp2} as 
\begin{equation}\label{eq:decomp3}
\dE_N [\MM] = \frac{1}{N^{\gamma+\frac{(p-1)n}{2}}}\sum_{\PP,\sigma} \sum_{i \in \INT{N}^{E(\MM_{\PP,\sigma})}}  \delta'_{\PP}(\psi(i)) \frac{\mu_{\PP}(\psi(i)) }{n_{\PP}(\psi(i))} \MM_{\PP,\sigma} (\cA^N_0,i),
\end{equation}
where $\MM_{\PP,\sigma} \in \widehat \cM_0(\cI_0)$ is the hyper-map obtained from $\MM$ by forming hyper-edges, for each block $b = \{u_1,\ldots,u_l\}$ of $\PP$, between the elements $\alpha(\sigma_{u_k}(\partial u_k))$, $k\in \INT{l},$ and associated to a variable in $\cA^N_0$ which are forced to coincide by the constraint: $\delta_{\PP,\sigma}(i) = 1$  for all $i$. Finally $\psi :  \INT{N}^{E(\MM_{\PP,\sigma})} \to \INT{N}^{E(\MM)}$ is the application  implied by the equality of some coordinates in the vector $i = (i_e)_{e \in E(\MM)}$.

We now remove the constraint over the set of $i$'s by using the fact that $ \delta'_{\PP}(\psi(i))$, $\mu_{\PP}(\psi(i))$ and $n_{\PP}(\psi(i))$ depend on $i$ only through its $\underset{N}{\sim}$-equivalence class. Such an equivalence class is encoded by a partition $\QQ$ on $E(\MM_{\PP,\sigma})$ whose blocks identify the coordinates of $i = (i_e)_{e \in E(\MM_{\PP,\sigma})}$ which are equal. We denote $i \in \QQ$ if the blocks of $i$ are given by the partition $\QQ$. In other words, we have
$$
 \delta'_{\PP}(\psi(i)) \frac{\mu_{\PP}(\psi(i)) }{n_{\PP}(\psi(i))} = \sum_{\QQ} f_{\PP} (\QQ) \IND ( i \in \QQ),
$$
where the sum runs over all partitions $\QQ$ of $E(\MM_{\PP,\sigma})$ and $f_{\PP} (\QQ)$ is the common value  for all $i \in \QQ$ of $\delta'_{\PP}(\psi(i)) \mu_{\PP}(\psi(i))  / n_{\PP}(\psi(i)) $.
We may now use the Moebius inversion formula on the poset of partitions of $E(\MM_{\PP,\sigma})$. From \cite[Proposition 3.7.1]{zbMATH06016068}, for a given pair $(\PP,\sigma)$, we get 
$$
 \sum_{i \in \INT{N}^{E(\MM_{\PP,\sigma})}}  \delta'_{\PP}(\psi(i)) \frac{\mu_{\PP}(\psi(i)) }{n_{\PP}(\psi(i))} \MM_{\PP,\sigma} (\cA^N_0,i) =  \sum_{\QQ} g_{\PP}(\QQ) \sum_{i \in \INT{N}^{E(\MM_{\PP,\sigma,\QQ})}} \MM_{\PP,\sigma,\QQ}(\cA^0_N,i),
$$
where the sum runs over all partitions $\QQ$ of $E(\MM_{\PP,\sigma})$, $\MM_{\PP,\sigma,\QQ}\in \widehat \cM_0(\cI_0)$ is the hyper-map obtained from $\MM_{\PP,\sigma}$ by gluing the hyper-edges of $\MM_{\PP,\sigma}$ in each block $\QQ$ and $$
g_{\PP}(\QQ) = \sum_{\QQ' \geq \QQ} M_{\PP,\sigma}(\QQ,\QQ') f_{\PP}(\QQ'), 
$$
with $M_{\PP,\sigma}(\QQ,\QQ')$ being the Moebius function of the poset. We thus have checked that \eqref{eq:decomp3} can be finally written as:
\begin{eqnarray}
\dE_N [\MM] & =  &  \frac{1}{N^{\gamma+\frac{(p-1)n}{2}}} \sum_{\PP,\sigma,\QQ} \sum_{i \in \INT{N}^{E(\MM_{\PP,\sigma,\QQ})}}  g_\PP(\QQ) \MM_{\PP,\sigma,\QQ} (\cA^N_0,i) \nonumber \\
& = &  \sum_{\PP,\sigma,\QQ} N^{\gamma(\PP,\sigma,\QQ) - \gamma - \frac{(p-1)n}{2} } g_\PP(\QQ) \MM_{\PP,\sigma,\QQ} (\cA^N_0),\label{eq:gdhede}
\end{eqnarray}
where $\gamma(\PP,\sigma,\QQ)$ is the number of connected components of $\MM_{\PP,\sigma,\QQ}$.

It remains to identify the leading terms in \eqref{eq:gdhede}. First, from assumption $(X1)$, $ |g_\PP(\QQ)| = O(1)$. Also, if $\gamma(\PP,\sigma)$ is the number of connected components of $\MM_{\PP,\sigma}$, we have 
$
\gamma(\PP,\sigma,\QQ) \leq \gamma(\PP,\sigma),
$
since $\MM_{\PP,\sigma,\QQ}$ is obtained from $\MM_{\PP,\sigma}$ by gluing some hyper-edges. We next claim that 
\begin{equation}\label{eq:fcla} \gamma(\PP,\sigma) \leq \gamma + \frac{(p-1)n}{2},
\end{equation}
with equality only if $|\PP|=n/2$, that is, only if $\PP$ is a pairing.

To this end, we compare $\MM_{\PP,\sigma}$ with $\MM$ one block of $\PP$ at a time: at each step we take a new block, say $b=\{u_1,\dots,u_l\}$, remove these vertices and replace them by an hyper-edge of size $l$ joining $\alpha(\sigma_{u_k} (\partial u_k)_j)$, $k\in\INT{l}$. Let $\NN_1$ be the current map before removing block $b$ and $\NN_2$ the one after.
We denote by  $C$ the set of connected components of $\NN_1 \setminus b$ which are incident to $b$ and let
$(c_{k,j})_{k\in\INT{l},\,j \in \INT{p}}$ be the array where $c_{k,j} \in C$ is the component
containing $\alpha(\sigma_{u_k}(\partial u_k)_j)$. In $\NN_1$, the vertex $u_k$ merges the $k$-th
row $\{c_{k,1},\dots,c_{k,p}\}$, while in $\NN_2$ the $j$-th hyper-edge merges the column $\{c_{1,j},\dots,c_{l,j}\}$. In other words, if $\Pi_1$ and
$ \Pi_2$ are the partitions of $C$ generated by the rows and the columns $(c_{k,j})_{k,j}$ respectively, we have $\gamma(\NN_2) - \gamma(\NN_1) = |\Pi_2| - |\Pi_1|$  where $\gamma(\NN_i)$ is the number of connected components of $\NN_i$.

Now, observe that $|\Pi_2| \leq p$ since  the number of connected components in an hypergraph is at most the number of hyper-edges if no vertex is isolated. We deduce that
\[
  \gamma(\NN_2)-\gamma(\NN_1)= |\Pi_2|-|\Pi_1|\le p-1 .
\]
Summing over the $|\PP|$ blocks and using that every block has size at least $2$, we obtain finally
\[
  \gamma(\PP,\sigma)-\gamma \leq  (p-1) |\PP| \leq \frac{(p-1)n}{2} ,
\]
with equality only if $|\PP|=n/2$, that is, only if $\PP \in \cP(V_{\SC})$. This proves \eqref{eq:fcla}.

We may now come back to \eqref{eq:gdhede}. Using assumption (A2) and  \eqref{eq:fcla}, we get 
$$
\dE_N [\MM] = \sum_{\PP \in \cP(V_{\SC}), \sigma , \QQ}  N^{\gamma(\PP,\sigma,\QQ) - \gamma - \frac{(p-1)n}{2} } g_\PP(\QQ) \MM_{\PP,\sigma,\QQ} (\cA^N_0) + O( \frac 1 N).
$$
The final observation is that $g_{\PP}(\QQ)$ is a function of $(f_{\PP}(\QQ'))_{\QQ' \geq \QQ}$ and hence $g_{\PP}(\QQ)$ depends on the distribution of the  $X_i$'s only through the possible values of $\mu_{\PP}(i)$. However, when $\PP \in \cP(V_{\SC})$ is a pairing, 
$ \mu_{\PP}(i) $ depends only on the second moment of the variables $X_i$'s. Since the second moments are identical under $\dE^N$ and $\dE_N^{\gauss}$, we finally get 
$$
\left| \dE_N [\MM] - \dE^{\gauss}_N [\MM] \right| = O( \frac{1}{N}). 
$$
This implies \eqref{eq:tbp1} and concludes the proof of Theorem \ref{th:free1b}. \qed


        

\bibliographystyle{abbrv}
\bibliography{bib}

\end{document}